\documentclass[11pt]{article}

\usepackage[english]{babel}
\usepackage{upref,amsfonts,amsxtra,latexsym,color, graphs,graphicx}
\usepackage{amssymb,amsmath,amsthm,epsf,mathrsfs,verbatim,hyperref,mathabx}
\usepackage[all,color, matrix, arrow]{xy}
 \usepackage[left=2.9cm,top=2.5cm,right=2.9cm,bottom=2cm,bindingoffset=0.5cm]{geometry}
\newtheorem{theorem}{Theorem}[section]
\newtheorem{lemma}[theorem]{Lemma}
\newtheorem{corollary}[theorem]{Corollary}
\newtheorem{proposition}[theorem]{Proposition}

\newtheorem{conjecture}[theorem]{Conjecture}

\theoremstyle{definition}
\newtheorem{definition}[theorem]{Definition}
\newtheorem{remark}[theorem]{Remark}
\newtheorem{example}[theorem]{Example}
\newtheorem{question}[theorem]{Question}

\numberwithin{equation}{section}
\numberwithin{theorem}{section}

\newcommand{\Ker}{\mathrm{Ker}}

\newcommand{\Id}{\mathrm{id}}

\newcommand{\Span}{\mathrm{span}}

\newcommand{\cU}{{\cal U}}
\newcommand{\cP}{{\cal P}}

\newcommand{\cO}{{\cal O}}

\newcommand{\cF}{{\mathcal F}}
\newcommand{\C}{{\mathbb C}}
\newcommand{\Z}{{\mathbb Z}}
\newcommand{\N}{{\mathbb N}}

\newcommand{\T}{{\mathbb T}}

\newcommand{\I}{{\mathbb I}}

\newcommand{\cK}{{\cal K}}
\newcommand{\cG}{{\mathcal G}}

\newcommand{\Cs}{{$C^*$-al\-ge\-bra}}

\newcommand{\Prim}{\mathrm{Prim}}
\newcommand{\Ideal}{\mathrm{Ideal}}

\newcommand{\sh}{{$^*$-ho\-mo\-mor\-phism}}

\newcommand{\bbullet}{{\color{blue}{\bullet}}} 
\newcommand{\rbullet}{{\color{red}{\bullet}}} 
\newcommand{\redarrow}{\ar@[red]@{-}}
\newcommand{\bluearrow}{\ar@[blue]@{-}}

\date{}
\begin{document} 

\title{Just-infinite $C^*$-algebras}

\author{Rostislav Grigorchuk\thanks{The first named author was supported by NSF grant DMS-1207699 and NSA grant H98230-15-1-0328}, Magdalena Musat$^\dagger$ and Mikael R\o rdam\thanks{The second  and third named authors were supported  by the Danish National Research Foundation (DNRF) through the Centre for Symmetry and Deformation at University of Copenhagen, and The Danish Council for Independent Research, Natural Sciences.}}

\maketitle

\centerline{\emph{Dedicated to Efim Zelmanov on the occasion of his 60th birthday}}

\begin{abstract} By analogy with the well-established notions of just-infinite groups and just-infinite (abstract) algebras, we initiate a systematic study of just-infinite \Cs s, i.e., infinite dimensional \Cs s for which all proper quotients are finite dimensional. We give a classification of such \Cs s in terms of their primitive ideal space, that leads to a trichotomy. We show that  just-infinite, residually finite dimensional \Cs s do exist by giving an explicit example of (the Bratteli diagram of) an AF-algebra with these properties. 

Further, we discuss when \Cs s and $^*$-algebras associated with a discrete group are just-infinite. If $\cG$  is the Burnside-type group of intermediate growth discovered by the first-named author, which is known to be just-infinite, then its group algebra $\C[\cG]$ and  its group \Cs{} $C^*(\cG)$ are not just-infinite. Furthermore, we show that the algebra $B = \pi(\C[\cG])$ under the Koopman representation $\pi$ of $\cG$ associated with its canonical action on a binary rooted tree is just-infinite. It remains an open problem whether the residually finite dimensional \Cs{} $C^*_\pi(\cG)$ is just-infinite.
\end{abstract}

\section{Introduction}

\noindent A group is said to be just-infinite if it is infinite and all its proper quotients are finite. Just-infinite groups arise, e.g., as branch groups (including the Burnside-type group of intermediate growth discovered by the first named author, see \cite{Gr80}). A trichotomy describes the possible classes of just-infinite groups, see \cite[Theorem 3]{Gr00}.  Each finitely generated infinite group has a just-infinite quotient. Therefore,  if we are interested in finitely  generated  infinite groups  satisfying a certain exotic  property  preserved by homomorphic  images, if such a group exists, then one is also to be found in  the  class  of  just-infinite  groups. 

The purpose of this paper is to investigate just-infinite dimensional \Cs s, defined to be infinite dimensional \Cs s for which all proper quotients by closed two-sided ideals are finite dimensional. (In the future, we shall omit ``dimensional'' and refer to these \Cs s as \emph{just-infinite}. The well-established notion of infiniteness of a unital \Cs, that is, its unit is Murray-von Neumann equivalent to a proper subprojection, is unrelated to our notion of just-infiniteness.) 
Analogous to just-infiniteness in other categories, any infinite dimensional \emph{simple} \Cs{} is just-infinite for trivial reasons. It is also easy to see that if a \Cs{} $A$ contains a simple essential closed two-sided ideal $I$ such that $A/I$ is finite dimensional, then $A$ is just-infinite. (A closed two-sided ideal in a \Cs{} is \emph{essential} if it has non-zero intersection with every other non-zero closed two-sided ideal.)  Hence, e.g., any essential extension of the compact operators on an infinite dimensional separable Hilbert space by a finite dimensional \Cs{} is just-infinite.

We give in Theorem~\ref{thm:types} a classification of just-infinite \Cs s into three types, depending on their primitive ideal space. In more detail,  if $A$ is a separable just-infinite \Cs, then its primitive ideal space is homeomorphic to one of the T$_0$-spaces $Y_n$, $0 \le n \le \infty$, defined in Example~\ref{ex:Y_n}. The case $n=0$ corresponds to $A$ being simple, while the case $1 \le n < \infty$ occurs when  $A$ is an essential extension of a simple \Cs{} by a finite dimensional \Cs{} with $n$ simple summands. If the primitive ideal space of a separable just-infinite \Cs{} $A$ is infinite, then it is homeomorphic to the T$_0$-space $Y_\infty$, and in this case $A$ is residually finite dimensional (i.e., there is a separating family of finite dimensional representations of $A$). The \Cs{}  $A$ has an even stronger property,  described in Section~2, that we call \emph{strictly residually finite dimensional}. We refer the reader to the survey paper \cite{BekLou-2000} for a more comprehensive treatment of residually finite groups and residually finite dimensional group \Cs s. 

To our knowledge, residually finite dimensional, for short RFD, just-infinite \Cs s have not been previously considered in the literature. A priori it is not even clear that they exist. This issue is settled in Section~4, where  we construct  a RFD just-infinite unital AF-algebra, by giving an explicit description of its Bratteli diagram. Residually finite dimensional \Cs s have been studied extensively, see for example \cite{Exel-Loring:RFD}, \cite{Dad:RFD}, \cite{Dad:non-exact},  \cite{BekLou-2000}, \cite{BO-approx}. They are always quasidiagonal (see, e.g., \cite{Brown:qd} or \cite[Chapter 7]{BO-approx}). Interesting classes of \Cs s, such as the full group \Cs s of the free groups and subhomogenous \Cs s, are RFD. Among RFD \Cs s, the just-infinite ones are distinguished by having the smallest possible ideal lattice.

In Section 5, we show that unital, separable, RFD just-infinite \Cs s need not be AF-algebras, nor nuclear, or even exact.  Using a construction of Dadarlat from \cite{Dad:non-exact}, we show that the just-infinite, residually finite dimensional AF-algebra constructed in Section~4 contains a RFD just-infinite,  non-nuclear sub-\Cs. Moreover, this AF-algebra is contained in a non-exact \Cs, which, likewise, is RFD and just-infinite.

Just-infiniteness for \Cs s is less prevalent than the corresponding property in the category of groups. Not every infinite dimensional \Cs{} has a just-infinite quotient, since, for example, no abelian \Cs{} is just-infinite; cf.\ Example~\ref{ex:ji-abelian}. There seems to be no natural condition ensuring that a \Cs{} has a just-infinite quotient. 

We discuss in Section~6 when a group \Cs{} is just-infinite, depending on properties of the group. We prove that the full group \Cs{} $C^*(G)$ of a discrete group $G$ is just-infinite if and only if its group algebra $\C[G]$ has a unique (faithful) $C^*$-norm and it is $^*$-just-infinite, i.e., is just-infinite as a $^*$-algebra. The former property holds trivially for any locally finite group. We do not know of any non-locally finite group for which $\C[G]$ has unique $C^*$-norm. We show that there are locally finite just-infinite groups $G$,  for which the group \Cs{} $C^*(G)$ and the group algebra $\C[G]$ are just-infinite. If the reduced group \Cs{} $C^*_\lambda(G)$ is just-infinite, then either $C^*_\lambda(G)$ must be simple, or $G$ must be amenable, in which case $C^*_\lambda(G)$ coincides with the full group \Cs{} $C^*(G)$. It seems plausible that the group \Cs{} associated with  unitary representations other than the universal or the left-regular one might be just-infinite for a larger class of groups. 

If the group algebra $\C[G]$ of a group $G$ is $^*$-just-infinite, then $G$ must be just-infinite, but the converse does not hold. Indeed, we show in Section 7 (Theorem~\ref{thm:no-branch}) that $\C[G]$ is not $^*$-just-infinite whenever $G$ is a branch group, while there are many branch groups which are just-infinite, e.g., the group $\cG$ of intermediate growth mentioned above. We show that the image $B= \pi(\C[\cG])$ of $\C[\cG]$ under the Koopman representation $\pi$ of $\cG$, associated with the canonical action of $\cG$ on a binary rooted tree, is just-infinite. We leave open the question whether or not the $C^*$-completion $C^*_\pi(\cG)$ of $B$ is just-infinite. In the affirmative case, this would provide an example of a RFD just-infinite \Cs{} arising from a group. 

\section{Preliminaries} 

\noindent As we shall later describe just-infinite \Cs s in terms of their \emph{primitive ideal space}, and as the interesting cases of just-infinite \Cs s are those that are \emph{residually finite dimensional}, we review in this section the relevant background.

\subsection{The primitive ideal space of a $C^*$-algebra}

\noindent 
A \Cs{} $A$ is said to be \emph{primitive} if it admits a faithful irreducible representation on some Hilbert space. It is said to be \emph{prime} if, whenever $I$ and $J$ are closed two-sided ideals in $A$ such that $I \cap J=0$, then either $I=0$, or $J=0$. It is easy to see that every primitive \Cs{} is prime, and it is a non-trivial result that the converse holds for all separable \Cs s; cf.\ \cite[Proposition~4.3.6]{Ped:C*-aut}. However, there are (complicated) examples of non-separable \Cs s that are prime, but not primitive, see \cite{Weaver:non-primitive}.

A closed two-sided ideal $I$  in a \Cs{} $A$ is said to be \emph{primitive} if $I \ne A$ and $I$ is the kernel of an irreducible representation of $A$ on some Hilbert space. The \emph{primitive ideal space}, $\Prim(A)$, is the set of all primitive ideals in $A$.
A closed two-sided ideal $I$ of $A$ is primitive if and only if the quotient $A/I$ is a primitive \Cs. In particular, $0 \in \Prim(A)$ if and only if $A$ is primitive. 
The primitive ideal space is a T$_0$-space when equipped with the \emph{hull-kernel topology}, which is given as follows: the closure $\overline{\cF}$ of a subset $\cF \subseteq \Prim(A)$ consists of all ideals $I \in \Prim(A)$ which contain $\bigcap_{J \in \cF} J$. If $A$ is primitive, so that $0 \in \Prim(A)$, then $\overline{\{0\}} = \Prim(A)$. 
In the commutative case, the primitive ideal space  is the usual spectrum:  $\Prim(C_0(X))$ is homeomorphic to $X$, whenever $X$ is a locally compact Hausdorff space.  The following fact will be used several times in the sequel:

\begin{remark} \label{ex:finite-dim}
Each finite dimensional \Cs{} $A$ is (isomorphic to) a direct sum of full matrix algebras,
$$A \cong M_{k_1}(\C) \oplus M_{k_2}(\C) \oplus \cdots \oplus M_{k_n}(\C),$$
for some positive integers $n, k_1, k_2, \dots, k_n$. As each matrix algebra is simple, $\Prim(A)$ can be naturally identified with the set $\{1,2, \dots, n\}$, equipped with the discrete topology. The primitive ideal space is Hausdorff in this case. 
\end{remark}

\noindent A closed subset $F$ of a T$_0$-space $X$ is said to be \emph{prime} if, whenever $F'$ and $F''$ are closed subsets of $X$ such that $F \subseteq F' \cup F''$, then $F$ is contained in one of $F'$ and $F''$. The closure of any singleton is clearly prime. A \emph{spectral space} is a T$_0$-space for which the converse holds: each closed prime subset is the closure of a singleton.
The results listed in the proposition below can be found in \cite[Sect.\ 5.4]{Mur:C*}, or \cite[Sect.\ 4.3]{Ped:C*-aut}:

\begin{proposition} \label{prop:Prim1}
 Let $A$ be a \Cs.
\begin{enumerate}
\item If $A$ is unital, then $\Prim(A)$ is a compact\footnote{A (possibly non-Hausdorff) topolotical space is said to be compact if it has the Heine-Borel property: each open cover can be refined to a finite open cover. Sometimes this  property is referred to as quasi-compactness.} $T_0$-space.
\item Let $I \in \Prim(A)$. Then $\{I\}$ is closed in $\Prim(A)$ if and only if $I$ is a maximal proper ideal in $A$, i.e., if and only if the quotient $A/I$ is simple.
\item If $A$ is separable, then $\Prim(A)$ is a second countable spectral space.
\end{enumerate}
\end{proposition}

\noindent By Remark~\ref{ex:finite-dim}, the only finite dimensional \Cs s which are primitive are those which are isomorphic to full matrix algebras. Hence, the following holds:

\begin{proposition} \label{prop:singleton}
Let $A$ be a separable \Cs, and let $I \in \Prim(A)$ be such that $A/I$ is finite dimensional. Then $A/I \cong M_k(\C)$, for some $k\in \N$, and $\{I\}$ is closed in $\Prim(A)$. 
\end{proposition}

\noindent
A T$_0$-space $X$ is said to be \emph{totally disconnected} if there is a basis for its topology consisting of compact-open sets. If the projections in a \Cs{} $A$ separate its ideals, then $\Prim(A)$ is totally disconnected. In this situation, we have the following result, which will be discussed in more detail in Section~\ref{sec:just-inf-AF}:

\begin{theorem}[Bratteli--Elliott, \cite{BraEll:AF}] \label{thm:BraEll}
Let $X$ be a second countable, totally disconnected spectral space. Then $X$ is homeomorphic to $\Prim(A)$, for some separable AF-algebra $A$. If $X$ is compact, then $A$ can be taken to be unital.
\end{theorem}

\noindent  Recall that an AF-algebra is a \Cs{} which is the completion of an increasing union of finite dimensional sub-\Cs s. 

We end this section by recalling that there is a one-to-one correspondence between open subsets $\cU$ of $\Prim(A)$ and closed two-sided ideals $I(\cU)$ of $A$, given by
\begin{equation} \label{eq:U-I}
I(\cU)  = \bigcap_{J \in \Prim(A) \setminus \cU} J,
\end{equation}
with the convention that $I(\emptyset) = 0$ and $I(\Prim(A)) = A$. Moreover, 
\begin{equation} \label{eq:X-U}
\Prim(A/I(\cU)) = \Prim(A) \setminus \cU,
\end{equation} 
for each open subset $\cU$ of $\Prim(A)$. Consequently, each closed subset of $\Prim(A)$ is the primitive ideal space of a quotient of $A$ (see \cite[Theorem 4.1.3]{Ped:C*-aut}).
Note furthermore that if $J \in \Prim(A)$, then $J=I(\cU)$, where $\cU$ is the complement of the closure of $\{J\}$.

\subsection{Residually finite dimensional \Cs s} \label{sec:RFD}

\noindent This section is devoted to discussing residually finite dimensional \Cs s and their primitive ideal spaces. We also introduce the class of so-called \emph{strictly residually finite dimensional} \Cs s, and describe them in terms of their primitive ideal space. 

A \Cs{} $A$ is said to be \emph{residually finite dimensional} (RFD),  if it admits a separating family of finite dimensional representations. The finite dimensional representations can be taken to be irreducible and pairwise (unitarily) inequivalent. (We note that two irreducible finite dimensional representations are equivalent if and only if they are \emph{weakly equivalent}, i.e., they have the same kernel.)

Assume that $\{\pi_i\}_{i \in \mathbb{I}}$ is a  family of irreducible and pairwise inequivalent finite dimensional representations of a \Cs{} $A$. 
Let $k_i$ be the dimension of the representation $\pi_i$, and identify the image of $\pi_i$ with $M_{k_i}(\C)$. We then get a \sh{}
$$\Psi_{\mathbb{I}} = \bigoplus_{i \in \mathbb{I}} \pi_i \colon A \to \prod_{i \in \mathbb{I}} M_{k_i}(\C).$$
Note that $\Psi_{\mathbb{I}}$ is injective if and only if $\bigcap_{i \in \mathbb{I}} \Ker(\pi_i) = \{0\}$, which again happens if and only if $\{\Ker(\pi_i) : i \in \mathbb{I}\}$ is a dense subset of $\Prim(A)$. Therefore, the following lemma holds; cf.\ Proposition~\ref{prop:singleton}:

\begin{lemma} \label{lm:RFD}
A \Cs{} $A$ is RFD if and only if $\Prim(A)$ contains a dense subset $\cP$ such that $A/I$ is a full matrix algebra, for each $I \in \cP$.
\end{lemma}

\noindent
Observe that $\{I\}$ is closed in $\Prim(A)$, for each $I \in \cP$, by Proposition~\ref{prop:Prim1}~(ii). If $A$ is separable, then one can choose the set $\cP$ in the lemma above to be countable.

Since the ideals $\Ker(\pi_i)$ are maximal and pairwise distinct (by the assumed inequivalence of the finite dimensional representations $\pi_i$, which implies that they are also weakly inequivalent), it follows from the Chinese Remainder Theorem that the map
\begin{equation*} 
\Psi_F = \bigoplus_{i \in F} \pi_i \colon A \to \prod_{i \in F} M_{k_i}(\C)
\end{equation*}
is surjective, for each finite subset $F$ of $\mathbb{I}$. 

\begin{definition} \label{def:SRFD} A unital \Cs{} $A$ is said to be \emph{strictly residually finite dimensional} (\emph{strictly RFD}) if there exists an infinite family $\{\pi_i \colon A \to M_{k_i}(\C)\}_{i \in \mathbb{I}}$ of irreducible, pairwise inequivalent, finite dimensional representations of $A$ such that the map
\begin{equation} \label{eq:Psi_T}
\Psi_T = \bigoplus_{i \in T} \pi_i \colon A \to \prod_{i \in T} M_{k_i}(\C)
\end{equation}
is injective, for each infinite subset $T$ of $\mathbb{I}$. 
\end{definition}

\noindent The following  characterizes strictly RFD \Cs s in terms of their primitive ideal space:

\begin{proposition} \label{prop:SRFD}
A unital separable \Cs{} $A$ is strictly RFD if and only if there exists an infinite subset $\cP$ of $\Prim(A)$ such that each of its infinite subsets is dense in $\Prim(A)$, and such that $A/I$ is finite dimensional, for each $I \in \cP$.
\end{proposition}

\noindent Note that if such a subset $\cP$ of $\Prim(A)$ exists, then each infinite subset of $\cP$ has the same properties,  and hence one can take $\cP$ to be countably infinite. 

\begin{proof} Suppose first that $A$ is a strictly RFD unital separable \Cs{}  witnessed by  an infinite family $\{\pi_i \colon A \to M_{k_i}(\C)\}_{i \in \mathbb{I}}$ of irreducible, pairwise inequivalent, finite dimensional representations. Set
$$\cP = \{\Ker(\pi_i) : i \in \mathbb{I}\} \subseteq \Prim(A).$$
If $I = \Ker(\pi_i ) \in \cP$, then $A/I \cong \pi_i(A)$ is finite dimensional. 
Let $T$ be an infinite subset of $\mathbb{I}$, then 
$0=\Ker(\Psi_T) = \bigcap_{i \in T} \Ker(\pi_i)$.
Therefore $\{\Ker(\pi_i) : i \in T\}$ is dense in $\Prim(A)$. 

Suppose conversely that $A$ is a unital separable \Cs{} for which there exists an infinite subset $\cP = \{I_i\}_{i \in \mathbb{I}}$ of $\Prim(A)$ satisfying the hypotheses. For each $i \in \mathbb{I}$, find an irreducible representation $\pi_i \colon A \to B(H_i)$ with $\Ker(\pi_i) = I_i$. Then $\pi_i(A) \cong A/I_i$ is finite dimensional, so $H_i$ is finite dimensional and $\pi_i(A) = B(H_i) \cong M_{k_i}(\C)$, where $k_i = \dim(H_i)$. Let $T$ be an infinite subset of $\mathbb{I}$. The associated map  $\Psi_T$ then satisfies 
$$\Ker(\Psi_T) = \bigcap_{i \in T} \Ker(\pi_i) = \bigcap_{i \in T} I_i = 0,$$
by the assumption that $\{I_i \}_{i \in T}$ is dense in $\Prim(A)$. This shows that $A$ is strictly RFD.
\end{proof}

\section{Just-infinite \Cs s: A classification result} \label{sec:classification}

\noindent By analogy with the notion of just-infiniteness in the category of groups and of abstract algebras, see \cite{McCarthy:ji}, we define a \Cs{} to be just-infinite as follows:

\begin{definition} A \Cs{} $A$ is said to be \emph{just-infinite} if it is infinite dimensional, and for each non-zero closed two-sided ideal $I$ in $A$, the quotient $A/I$ is finite dimensional.
\end{definition}

\begin{lemma} \label{lm:prime}
Every just-infinite \Cs{} is prime.
\end{lemma}

\begin{proof} Let $A$ be a  just-infinite \Cs, and let $I$ and $J$ be two non-zero closed two-sided ideals in $A$. Consider the natural homomorphism $\pi \colon A \to A/I \oplus A/J$.
By the assumption that $A$ is just-infinite, the image is finite dimensional. It follows that $\pi$ cannot be injective, so $I \cap J = \Ker(\pi) \ne 0$.
\end{proof}

\begin{example} \label{ex:ji-abelian}
The group $\Z$ is just-infinite, and it is the only abelian just-infinite group. It is also known, see \cite[Proposition 3(a)]{Gr00}, that every  finitely generated infinite group has a just-infinite quotient. 

The corresponding statements for \Cs s are false:
No commutative \Cs{} is just-infinite, since no commuative \Cs{} other than $\C$ is prime.  This also shows that no commutative \Cs{} has a just-infinite quotient. 

It is  well-known that every \emph{unital} \Cs{} has a maximal proper closed two-sided ideal, and hence a quotient which is simple. If, moreover, such a simple quotient is infinite dimensional, then it is just-infinite. There seems to be no satisfactory description of unital \Cs s having an infinite dimensional simple quotient. 
\end{example}

\begin{lemma} \label{lm:essential} Each non-zero closed two-sided ideal in a just-infinite \Cs{} is essential and infinite dimensional.
\end{lemma}

\begin{proof} It is easy to see that a \Cs{} is prime if and only if each non-zero closed two-sided ideal is essential, so the first statement of the lemma follows from Lemma~\ref{lm:prime}. 

If a closed two-sided ideal $I$ in a \Cs{} $A$ has a unit $e$, then $e$ is a central projection in $A$ and $I = Ae$. Thus $Ae$ and $A(1-e)$ are closed two-sided ideals in $A$ with zero intersection. So if $I$ is essential, then $A(1-e) = 0$ and $I=A$. Now, as each finite dimensional \Cs{} has a unit, we see that no non-zero closed two-sided ideal in a just-infinite \Cs{} can be finite dimensional.
\end{proof}

\noindent The class of just-infinite \Cs s does not have good permanence properties. In fact, almost all natural operations on \Cs s (such as passing to sub-\Cs s, extensions, passing to ideals and quotients, taking inductive limits, Morita equivalence, forming crossed products by suitable groups) fail to be consistent with the class of just-infinite \Cs s. However, the following permanence-type properties of just-infinite \Cs s do hold:

\begin{proposition} \label{permanence} \mbox{}
\begin{enumerate} 
\item If $B$ is an infinite dimensional hereditary\footnote{A sub-\Cs{} $B$ of a \Cs{} $A$ is \emph{hereditary} if whenever $b \in B$ and $a \in A$ are such that $0 \le a \le b$, then $a \in B$.} sub-\Cs{} of a just-infinite \Cs{} $A$, then $B$ is just-infinite. In particular, each non-zero closed two-sided ideal in a just-infinite \Cs{} is again just-infinite. 
\item Let $0 \to I \to A \to Q \to 0$ be a short exact sequence of \Cs s, where $I$ is an essential ideal in $A$. Then $A$ is just-infinite if and only if $Q$ is finite dimensional and $I$ is just-infinite.
\end{enumerate}
\end{proposition}

\begin{proof} (i). Let $J$ be a non-zero closed two-sided ideal in $B$, and let $I$ be the (necessarily non-zero) closed two-sided ideal in $A$ generated by $J$. Then $J = B \cap I$, so $B/J$ is isomorphic to a (hereditary) sub-\Cs{} of $A/I$. The latter is finite dimensional, so $B/J$ must also be finite dimensional. 

The second part of (i) follows from the fact that each closed two-sided ideal in a \Cs{} is a hereditary sub-\Cs, together with Lemma~\ref{lm:essential}, which ensures that each non-zero closed two-sided ideal in $A$ must be infinite dimensional. 

(ii). Suppose that $I$ is just-infinite and $Q$ is finite dimensional. Let $J$ be a non-zero ideal in $A$. Then we have a short exact sequence $0 \to I/(I \cap J) \to A/J \to Q/\pi(J) \to 0$, where $\pi \colon A \to Q$ is the quotient mapping. Now, $I \cap J$ is non-zero (since $I$ is an essential ideal), so $I/(I \cap J)$ is finite dimensional. This implies that $A/J$ is finite dimensional, being an extension of two finite dimensional \Cs s.

Conversely, if $A$ is just-infinite, then $Q$, which is isomorphic to the quotient $A/I$, is finite dimensional (because $I$ is non-zero). The ideal $I$ cannot be finite dimensional (since otherwise $A$ would be finite dimensional), so it follows from (i) that $I$ is just-infinite.
\end{proof}

\noindent The observation made above that the class of just-infinite \Cs s is not closed under Morita equivalence,  can be justified as follows. If $A$ is a just-infinite \Cs{} and if $\cK$ denotes the \Cs{} of compact operators on a separable Hilbert space, then $A \otimes \cK$ is just-infinite if and only if $A$ is simple (since all proper non-zero quotients of $A \otimes \cK$ are stable, and therefore infinite dimensional). 

\begin{remark}[Hereditary just-infiniteness]  If $G$ is a residually finite group and all its normal subgroups of finite index, including $G$ itself, are just-infinite, then $G$ is said to be \emph{hereditarily just-infinite}. Just-infinite branch groups are residually finite, but not hereditarily just-infinite; cf.\ \cite[Section 6]{Gr00}, so a finite index normal subgroup of a just-infinite group need not be just-infinite. We shall say more about hereditarily just-infinite groups and just-infinite branch groups in Examples~\ref{ex:PSLnZ} and \ref{ex:Dinfty} and in Theorem~\ref{thm:no-branch}.

It follows from Proposition~\ref{permanence} above that just-infinite \Cs s automatically have a property analogous to being hereditarily just-infinite for groups: Any non-zero closed two-sided ideal in a just-infinite \Cs{} is itself just-infinite. Note also that the following three conditions for a closed two-sided ideal $I$ in a just-infinite \Cs{} $A$ are equivalent (cf.\ Lemma~\ref{lm:essential} and the definition of being just-infinite): $I$ is non-zero, $I$ is infinite dimensional, and $I$  has finite co-dimension in $A$.
\end{remark}

\noindent We proceed to describe the primitive ideal space of a just-infinite \Cs. They turn out to be homeomorphic to one of the T$_0$-spaces in the following class:

\begin{example} \label{ex:Y_n}
For each  $n \in \{0,1,2, \dots, \infty\}$, consider the T$_0$-space $Y_n$ defined to be the disjoint union $Y_n = \{0\} \cup Y_n'$, where $Y_n'$ is a set with $n$ elements, if $n$ is finite,  and $Y'_n$ has  countably infinitely many elements, if $n = \infty$.
Equip $Y_n$ with the topology  for which the closed subsets of $Y_n$ are precisely the following sets: $\emptyset$, $Y_n$, and all finite subsets of $Y_n'$. 

We shall usually take $Y_n'$ to be  $\{1,2, \dots, n\}$, if $1 \le n < \infty$, and $Y_\infty'$ to be  $\N$.  
\end{example}

\noindent The spaces $Y_n$ have the following axiomatic properties:

\begin{lemma} \label{lm:Y_n}
A (non-empty) second countable T$_0$-space $X$ is homeomorphic to $Y_n$, for some $n \in \{0,1,2, \dots, \infty\}$,  if and only if it the following conditions hold, for some point $x_0 \in X$:
\begin{enumerate}
\item[\rm{(A)}] $\{x_0\}$ is dense in $X$,
\item[\rm{(B)}] $\{x\}$ is closed, for all $x \in  X \setminus \{x_0\}$,
\item[\rm{(C)}] each infinite subset of $X$ is dense in $X$.
\end{enumerate}
Moreover, if $X$ is any T$_0$-space satisfying conditions (A), (B) and (C) above, then 
\begin{enumerate}
\item the closed subsets of $X$ are the following sets: $\emptyset$, $X$, and all finite subsets of $X \setminus \{x_0\}$,
\item $X$ is second countable if and only if $X$ is countable,
\item each subset of $X$ is compact (in particular, $X$ is totally disconnected),
\item  $X$ is a spectral space.\footnote{See definition above Proposition~\ref{prop:Prim1}}
\end{enumerate}
\end{lemma}

\begin{proof} The spaces $Y_n$ satisfy conditions (A), (B), and (C) with $x_0 = 0$. We show below that (A), (B) and (C) imply (i)--(iv). Any second countable T$_0$-space $X$ satisfying (i) and (ii) is homeomorphic to $Y_n$, where $n$ is the cardinality of $X' = X \setminus \{x_0\}$. Indeed, $X$ is countable by (ii), and  any bijection $f \colon X \to Y_n$, with $f(x_0) = 0$, is  a homeomorphism by (i).

Let now $X$ be a T$_0$-space satisfying (A), (B) and (C). We show that (i), (ii), (iii) and (iv) hold. Set $X' = X \setminus \{x_0\}$. 
It follows from (B) that each finite subset of $X'$ is closed, and so are $\emptyset$ and $X$. Conversely, if $F$ is a closed subset of $X$ and if $F \ne X$, then $F$ must be a finite subset of $X'$ by (A) and (C). Hence (i) holds. 

Suppose now that $X$ is second countable and $|X| > 1$. Let $\{U_n\}_{n=1}^\infty$ be a basis for the topology on $X$ consisting of non-empty open sets. For $n\ge 1$, set $F_n = X \setminus U_n$, and observe that $F_n$ is finite (or empty) by (i). Let $x \in X'$. Then $X \setminus \{x\}$ is open by (i), so $U_n \subseteq X \setminus \{x\}$ for some $n$, whence $x \in F_n$. Thus $X'$ is contained in the countable set  $\bigcup_{n=1}^\infty F_n$,  so (ii) holds. 

Let $K$ be an arbitrary subset of $X$ and let $\{W_i\}_{i \in I}$ be an open cover of $K$. Take any $i_0 \in I$ such that $W_{i_0}$ is non-empty. Then the set $F= X \setminus W_{i_0}$ is finite. Hence $F \cap K$ is covered by finitely many open sets from the collection $\{W_i\}_{i \in I}$ which, together with $W_{i_0}$, gives a finite open cover of $K$. This proves (iii). 

Finally, let $F \ne \emptyset$ be a closed subset of $X$ which is prime. If $F= X$, then $F$ is the closure of $\{x_0\}$. If $F \ne X$, then $F$ is a finite subset of $X'$, by (i). Write $F = \bigcup_{x \in F} \{x\}$, and note that each singleton $\{x\}$, $x \in F$, is closed. Hence $F$ can have at most one element, so it is in particular  the closure of a singleton. This proves that $X$ is a spectral space. 
\end{proof}

\begin{lemma} \label{lm:Prim-ji} Let $A$ be a separable \Cs. The following hold:
\begin{enumerate}
\item $\Prim(A)$ is homeomorphic to $Y_n$, for some $n \in \{0,1,2, \dots, \infty\}$, if and only if  the following three conditions hold:
\begin{enumerate}
\item $A$ is primitive,
\item $A/I$ is simple, for each non-zero primitive ideal $I$ in $A$,
\item if $\Prim(A)$ is infinite, then $\bigcap_{I \in \cP} I = 0$, for each infinite subset $\cP$ of $\Prim(A)$. 
\end{enumerate}
\item If $\Prim(A)$ is infinite and $A$  satisfies {\rm{(b)}} and {\rm{(c)}}, then it automatically satisfies {\rm{(a)}}. If $A/I$ is finite dimensional, for each non-zero $I \in \Prim(A)$, then condition {\rm{(b)}} holds.
\item If $A$ is just-infinite, then $\Prim(A) = Y_n$, for some $n \in \{0,1,2, \dots, \infty\}$. 
\end{enumerate}
\end{lemma}

\begin{proof}   (i). It follows from Proposition~\ref{prop:Prim1}~(iii) that $\Prim(A)$ is second countable. It therefore suffices to show that the conditions (a), (b) and (c) are equivalent to items (A), (B) and (C) of Lemma~\ref{lm:Y_n} (with $X = \Prim(A)$ and $x_0 = 0$). By definition, $A$ is primitive if and only if $0 \in \Prim(A)$, so (a) is equivalent to (A). The equivalence of (b) and (B) follows from Proposition~\ref{prop:Prim1} (ii), while the equivalence of (c) and (C) follows from the fact that a subset $\cP$ of $\Prim(A)$ is dense  if and only if $\bigcap_{I \in \cP} I = 0$.

(ii). Suppose that $\Prim(A)$ is infinite and that (b) and (c) are satisfied. We assert that (a) holds, as well. By \cite[Proposition~4.3.6]{Ped:C*-aut} it suffices to check that $\Prim(A)$ is prime, i.e., whenever $\Prim(A) = F \cup G$, where $F$ and $G$ are closed subsets of $\Prim(A)$, then one of $F$ and $G$ is equal to  $\Prim(A)$. However,  one of $F$ and $G$ must be infinite, and hence dense in $\Prim(A)$ by (c) (which is equivalent to (C)), and therefore one of $F$ and $G$ must be equal to $\Prim(A)$. The remaining assertion follows from  Proposition~\ref{prop:singleton}.

(iii). Suppose that $A$ is just-infinite. To see that $\Prim(A) \cong Y_n$, for some $n$, it suffices by (i) and (ii) to check that  (b) and (c)  hold. Moreover, we conclude from (ii) that (b) holds because $A/I$ must be finite dimensional, for each non-zero ideal $I$. Suppose that $\cP$ is an infinite subset of $\Prim(A)$, and set $J = \bigcap_{I \in \cP} I$. For each finite subset $\cF$ of $\cP \setminus \{0\}$, let $J_\cF = \bigcap_{I \in \cF} I \supseteq J$. Then $A/J_\cF$ is isomorphic to $\bigoplus_{I \in \cF} A/I$, which has dimension at least $|\cF|$, so $A/J$ also has dimension at least $|\cF|$. As this holds for all finite subsets $\cF$ of  $\cP$, we conclude that  $A/J$ must be infinite dimensional.  Hence $J=0$, since $A$ is just-infinite. This proves that $\cP$ is dense in $\Prim(A)$, so (c) holds. 
\end{proof}

\noindent  Just-infinite \Cs s are classified as follows (to be compared with \cite[Theorem 3]{Gr00}):

\begin{theorem} \label{thm:types}
Let $A$ be a separable \Cs. Then $A$ is just-infinite if and only if $\Prim(A)$ is homeomorphic to $Y_n$, for some $n \in \{0,1,2, \dots, \infty\}$, and each non-faithful irreducible representation of $A$ is finite dimensional. (If $n=0$, we must also require that $A$ is infinite dimensional; this is automatic when $n \ge 1$.)  Moreover:
\begin{itemize}
\item[\rm{($\alpha$)}] $\Prim(A) = Y_0$ if and only if $A$ is simple.  Every infinite dimensional simple \Cs{} is just-infinite.
\item[\rm{($\beta$)}]  $\Prim(A) = Y_n$, for some integer $n \ge 1$, and $A$ is just-infinite, if and only if $A$ contains a simple non-zero essential infinite dimensional ideal $I_0$ such that $A/I_0$ is finite dimensional.  In this case, $n$ is equal to the number of simple summands of $A/I_0$.
\item[\rm{($\gamma$)}] The following conditions are equivalent:
\begin{enumerate}
\item $A$ is just-infinite and $\Prim(A) = Y_\infty$,
\item $A$ is just-infinite and RFD,
\item $\Prim(A)$ is an infinite set, all of its infinite subsets are dense, and $A/I$ is finite dimensional, for each non-zero $I \in \Prim(A)$,
\item $\Prim(A)$ is an infinite set, the direct sum representation $\bigoplus_{i \in T} \pi_i$ is faithful for each infinite family $\{\pi_i\}_{i \in T}$ of pairwise  inequivalent irreducible representations of $A$, and each non-faithful irreducible representation of $A$ is finite dimensional.
\end{enumerate}
\end{itemize}
\end{theorem}

\noindent We shall occasionally refer to a just-infinite \Cs{} as being of type ($\alpha$), ($\beta$) and  ($\gamma$), respectively, if it satisfies the corresponding condition in the theorem above. In view of the theorem, we shall also, more frequently, refer to a just-infinite \Cs{} of type ($\gamma$) as a \emph{RFD just-infinite} \Cs.

\begin{proof}  
If $A$ is just-infinite and separable, then $\Prim(A)=Y_n$, for some $n \in \{0,1,2 \dots, \infty\}$, by Lemma~\ref{lm:Prim-ji},  and each non-faithful irreducible representation of $A$ is finite dimensional (by the definition of being just-infinite). 

Suppose conversely that $\Prim(A) =  Y_n$, for some $n \in \{0,1,2, \dots, \infty\}$, and that each non-faithful irreducible representation of $A$ is finite dimensional. We show that $A$ then must be just-infinite (if it is also infinite dimensional). This is clear if $n=0$, since $A$ is simple in this case. This also shows that  ($\alpha$) holds.  

Suppose that $1 \le n \le \infty$. Since $Y_n$ is non-Hausdorff, when $n > 0$, and 
 the primitive ideal space of any finite dimensional \Cs{} is Hausdorff; cf.\ Remark~\ref{ex:finite-dim}, $A$ must be infinite dimensional. Write
\begin{equation} \label{eq:a}
\Prim(A) = \{0\} \cup \{I_j\}_{j=1}^n.
\end{equation}
Any non-zero proper ideal $J$ of $A$ is the intersection of the primitive ideals in $A$ that contain it. By  Lemma \ref{lm:Prim-ji} (c), any intersection of infinitely many distinct primitive ideals of $A$ is zero. Hence $J = \bigcap_{j \in F} I_j$, for some, necessarily finite, subset $F$ of $\N$ (or of $\{1,2, \dots , n\}$, if $n < \infty$). Now, $A/J$ is isomorphic to $\bigoplus_{j \in F} A/I_j$, and each quotient $A/I_j$ is finite dimensional by assumption, whence $A/J$ is finite dimensional. This shows that $A$ is just-infinite.

We proceed to verify the claims in ($\beta$) and ($\gamma$).

($\beta$). The ``if'' part follows from Proposition~\ref{permanence} (ii). Moreover, $\Prim(A)$ consists of $0$ (cf.\ Lemma~\ref{lm:prime}) and the kernels of the maps onto the $n$ simple summands of $A/I_0$, so $\Prim(A)$ has cardinality $n+1$. Also, $\Prim(A)$ is homeomorphic to $Y_k$, for some $k$, by Lemma~\ref{lm:Prim-ji} (iii), and by cardinality considerations, we conclude that $k=n$. 

Let us prove the ``only if'' part. Suppose that $A$ is just-infinite and $\Prim(A) = Y_n$, for some $n \in \N$.  Retain the notation set forth in \eqref{eq:a}, and let $I_0 = \bigcap_{j=1}^n I_j$. In the notation from \eqref{eq:U-I}, we have $I_0 = I(\{0\})$ (observe that $\{0\}$ is an open subset of $\Prim(A)$, when $n < \infty$). We deduce  that $I_0$ is non-zero and simple.  Each non-zero ideal in a primitive \Cs{} is essential, so  $I_0$ is an essential ideal in $A$, by  Lemma~\ref{lm:prime}. Since $A$ is just-infinite, $A/I_0$ is finite dimensional.  Finally, by \eqref{eq:X-U},
$$\Prim(A/I_0) = \Prim(A) \setminus \{0\} = \{I_1,I_2,\dots, I_n\},$$
and since $A/I_0$ is finite dimensional, $n$ is the number of direct summands of $A/I_0$; cf.\ Remark~\ref{ex:finite-dim}.

($\gamma$). (i) $\Rightarrow$ (iii). If $A$ is just-infinite, then $A/I$ is finite dimensional, for each non-zero ideal $I$ in $A$; and if $\Prim(A) = Y_\infty$, then each infinite subset of $\Prim(A)$ is dense (by Lemma~\ref{lm:Prim-ji} (i)(c)), and $\Prim(A)$ is an infinite set.

(iii) $\Rightarrow$ (ii). The assumptions in (iii) imply that $A$ is RFD; cf.\ Lemma~\ref{lm:RFD}.  If $\pi$ is a non-faithful irreducible representation of $A$, then $\Ker(\pi)=I$ is a non-zero primitive ideal in $A$, so $\pi(A) \cong A/I$ is finite dimensional. To conclude that $A$ is just-infinite we show that $\Prim(A)$ is homeomorphic to $Y_\infty$. For this it suffices to verify conditions (b) and (c)  of Lemma~\ref{lm:Prim-ji} (i). Item (b) holds because $A/I$ is finite dimensional, for each non-zero $I \in \Prim(A)$; cf.\  Lemma~\ref{lm:Prim-ji} (ii). Item (c) is equivalent to condition (C) in Lemma~\ref{lm:Y_n}, which holds by assumption. 

(ii) $\Rightarrow$ (i). If $A$ is RFD, then $A$ cannot be just-infinite of type ($\alpha$) or ($\beta$), so $\Prim(A)$ must be homeomorphic to $Y_\infty$.

(iii) $\Rightarrow$ (iv). We already saw that (iii) implies that $A$ is just-infinite, and hence that each non-faithful irreducible representation is finite dimensional. Let $\{\pi_i\}_{i \in T}$ be an infinite family of pairwise  inequivalent irreducible representations of $A$. 
Since  $\{\Ker(\pi_i) : i \in T\}$ is an infinite set, and hence by assumption a dense subset of $\Prim(A)$, it follows that 
the kernel of $\bigoplus_{i \in T} \pi_i$, which is equal to $\bigcap_{i \in T} \Ker(\pi_i)$, must be zero. 

(iv) $\Rightarrow$ (iii). Let $\cP$ be an infinite subset of $\Prim(A)$, and choose pairwise  inequivalent irreducible representations $\{\pi_i\}_{i \in T}$ of $A$ such that $\cP = \{\Ker(\pi_i) : i \in T\}$. The assumptions in (iv) now yield 
$$0 = \Ker\Big( \bigoplus_{i \in T} \pi_i\Big) = \bigcap_{i \in T} \Ker(\pi_i) = \bigcap_{I \in \cP} I,$$
which implies that $\cP$ is dense in $\Prim(A)$. 

If $I$ is a non-zero primitive ideal in $A$, then $I = \Ker(\pi)$, for some (non-faithful) irreducible representation of $A$, so $A/I \cong \pi(A)$ is finite dimensional.
\end{proof}

\noindent  The following result follows immediately from Proposition~\ref{prop:SRFD} and Theorem~\ref{thm:types}:

\begin{corollary} \label{cor:strictRFD}
Each separable RFD just-infinite \Cs{} is strictly RFD.
\end{corollary}

\noindent We note that not all strictly RFD \Cs s are just-infinite; cf.\ Section~\ref{subsec:SRFDvsJI} below. 

\begin{corollary} The primitive ideal space of a separable just-infinite \Cs{} is countable. Moreover, any RFD just-infinite separable \Cs{} has countably infinitely many equivalence classes of finite dimensional irreducible representations.
\end{corollary}

\begin{proof} The first claim follows from Lemma~\ref{lm:Prim-ji}~(iii). The second claim follows from Theorem~\ref{thm:types} ($\gamma$), by the fact that there is a one-to-one correspondence between weak equivalence classes of irreducible representations and the primitive ideal space of a separable \Cs{} (given by mapping an irreducible representation to its kernel), and by the fact, observed earlier, that two finite dimensional irreducible representations are unitarily equivalent if they are weakly equivalent.
\end{proof}

\begin{remark} It is shown in  Theorem~\ref{thm:types} that a separable \Cs{} $A$ is just-infinite if and only if  the following two conditions hold: $\Prim(A)$ $=$ $Y_n$, for some $n \in \{0,1,2, \dots, \infty\}$ and each non-faithful irreducible representation is finite dimensional. These two conditions are independent, i.e., none of them alone implies that $A$ is just-infinite, as shown below.

If $X$ is a Hausdorff space and $k$ is a positive integer, then all irreducible representations of $M_k(C(X))$ have dimension $k$, and $\Prim(M_k(C(X))) = X$. If $X$ is not a point, then $X$ is not homeomorphic to $Y_n$, for any $n$, because $Y_n$ is non-Hausdorff for all $n > 0$. Therefore $M_k(C(X))$ is not just-infinite.

For each $n \in \{0,1,2, \dots, \infty\}$, there is a unital AF-algebra whose primitive ideal space is $Y_n$,  by Theorem~\ref{thm:BraEll} and Lemma~\ref{lm:Y_n}. The AF-algebras obtained in this way  may or may not have the property that each non-faithful irreducible representation is finite dimensional. Tensoring such an AF-algebra by a UHF-algebra, we obtain a unital separable \Cs{} whose primitive ideal space is $Y_n$, and which has no finite dimensional irreducible representations. Therefore, it is not just-infinite. 

We show in the next example and in Section~\ref{sec:just-inf-AF} below that each  space $Y_n$ can be realized as the primitive ideal space of a just-infinite AF-algebra.
\end{remark}

\begin{example}[Existence of just-infinite \Cs s] \label{ex:exab}
Any simple infinite dimensional \Cs{} is just-infinite of type ($\alpha$) (and there are many examples of such, both in the unital and the non-unital case). 

To exhibit examples of just-infinite \Cs s of type ($\beta$), 
let $n \in \N$, and let $F$ be a finite dimensional \Cs{} with $n$ simple summands, e.g., $F = \C \oplus \cdots \oplus \C$ with $n$ summands. Let $H$ be an infinite dimensional separable Hilbert space, and let $\pi \colon B(H) \to B(H) / \cK$ be the quotient mapping onto the Calkin algebra, where as before $\cK$ denotes the compact operators on $H$. Let $\tau \colon F \to B(H)/\cK$ be a unital injective \sh. Set
\begin{equation} \label{eq:A}
A = \pi^{-1}\big(\tau(F)\big) \subseteq B(H).
\end{equation}
Then $\cK$ is a simple essential ideal in $A$ and $A/\cK$ is isomorphic to $F$. Hence $A$ is just-infinite of type ($\beta$), and $\Prim(A) = Y_n$; cf.\ Theorem~\ref{thm:types} ($\beta$). Since $A$ is an extension of two AF-algebras, it is itself an AF-algebra. 

Each just-infinite \Cs{} $A$ arising as in \eqref{eq:A} above is of \emph{type I}: for each irreducible representation of $A$ on a Hilbert space $H$, the image of $A$ contains the compact operators on $H$. Conversely, a separable \Cs{} $A$ of \emph{type I} is just-infinite if and only if it is isomorphic to the compact operators $\cK$ on a separable Hilbert space, or it is of the form described in \eqref{eq:A} above for some finite dimensional \Cs{} $F$. Indeed, if $A$ is separable, just-infinite and of type I, then $A$ is prime by Lemma~\ref{lm:prime}, hence primitive (because it is separable), so it admits a faithful irreducible representation $\rho$ on some (separable) Hilbert space. Being of type I, $\rho(A)$ contains the compact operators $\cK$. If $\rho(A) \ne \cK$, then  the quotient $B:=\rho(A)/\cK$ is finite dimensional, because $A$ is just-infinite, so $A \cong \rho(A) =\pi^{-1}(B)$ is as in \eqref{eq:A}.
\end{example}

\noindent 
It requires more work to establish the existence of RFD just-infinite \Cs s, i.e., those of type ($\gamma$). This will be done in Section \ref{sec:just-inf-AF}. 

\begin{remark}[Characteristic sequences of just-infinite \Cs s] \label{rem:add-structure}
Let $A$ be a unital separable just-infinite \Cs. If $A$ is non-simple, then $\Prim(A) = Y_n$, for some $n \in \{1,2, \dots, \infty\}$. Let $\{I_j\}_{j = 1}^n$ be the non-zero primitive ideals of $A$. Then $A/I_j \cong M_{k_j}(\C)$, for some $k_j \in \N$; cf.\ Proposition~\ref{prop:singleton}. The resulting $n$-tuple, or sequence, $\{k_j\}_{j=1}^n$ (as an unordered set) is an invariant of $A$, which we shall call the \emph{characteristic sequence of $A$}.

For each $j$, choose an irreducible representation $\pi_j \colon A \to M_{k_j}(\C)$ with kernel $I_j$. We say that such a  sequence $\{\pi_j\}_{j=1}^n$ is an \emph{exhausting sequence} of pairwise inequivalent non-faithful irreducible representations of $A$. Equivalently, $\{\pi_j\}_{j=1}^n$ is an exhausting sequence of pairwise inequivalent non-faithful irreducible representations of $A$ if 
$\Prim(A) \setminus \{0\} = \big\{ \Ker(\pi_j) : j = 1,2, \dots, n\big\}$,
and $\Ker(\pi_j) \ne \Ker(\pi_i)$ when $i \ne j$.

If $n \in \N$ and if $I_0$ is the (unique) simple essential ideal in $A$, then (as in the proof of Theorem \ref{thm:types})  we have the following isomorphisms
\begin{equation} \label{eq:beta}
A/I_0 \; \cong \; \bigoplus_{j=1}^n A/I_j \; \cong \; \bigoplus_{j=1}^n M_{k_j}(\C).
\end{equation}
It follows from Example \ref{ex:exab} (and Remark~\ref{ex:finite-dim}) that for all positive integers $k_1,k_2, \dots, k_n$, there exists a just-infinite \Cs{} $A$, which is necessarily an AF-algebra, such that \eqref{eq:beta} holds with $I_0 = \cK$. This argument shows in particular that each \emph{finite} characteristic sequence $\{k_j\}_{j=1}^n$, where $n \in \N$, is realized by a just-infinite AF-algebra (of type ($\beta$)). 
\end{remark}

\noindent We end this section by showing that the characteristic sequence $\{k_j\}_{j=1}^\infty$ of a RFD just-infinite  \Cs{} must tend to infinity. The proof of this fact involves results about subhomogeneous \Cs s. Recall that a \Cs{} is said to be \emph{subhomogeneous} if it is isomorphic to a sub-\Cs{} of $M_k(C(X))$, for some compact Hausdorff space $X$, and some $k \in \N$. The next proposition is well-known, but we include a brief proof for the sake of completeness.

\begin{proposition} \label{prop:subhom}
For a \Cs{} $A$, the following conditions are equivalent:
\begin{enumerate}
\item $A$ is subhomogeneous,
\item the bidual $A^{**}$ of $A$ is isomorphic to $\bigoplus_{j=1}^n M_{k_j}(C(\Omega_j))$, for some positive integers $n,k_1,k_2, \dots, k_n$, and some (extremally disconnected) compact Hausdorff spaces $\Omega_1$, $\Omega_2$, $\dots$, $\Omega_n$,
\item there exists a positive integer $k$ such that each irreducible representation of $A$ has dimension at most $k$,
\item there exist a positive integer $k$ and  a separating family $\{\pi_i\}_{i \in T}$ of irreducible representations of $A$ such that each $\pi_i$ has dimension at most $k$. 
\end{enumerate}
\end{proposition}

\begin{proof} The implication  (ii) $\Rightarrow$ (i) holds because $A$ is a sub-\Cs{} of  $A^{**}$. If (iii) holds, then $A^{**}$, which is a von Neumann algebra,  cannot have central summands of type $I_n$, for $n > k$, or of type II or III. Therefore (ii) holds.
The implication  (i) $\Rightarrow$ (iv) follows easily from the definition of subhomogeneity.
Suppose now that (iv) holds, and that there exists an irreducible representation of $A$ of dimension strictly greater than $k$ (possibly infinite dimensional). By (a version of) Glimm's lemma, see, e.g., \cite[Proposition 3.10]{RobRor:div}, there is a non-zero \sh{} $\rho \colon C_0((0,1]) \otimes M_{k+1} \to A$. However, there is no non-zero \sh{} $C_0((0,1]) \otimes M_{k+1} \to B(H)$ when $\mathrm{dim}(H) \le k$, so it follows that $\pi_i \circ \rho = 0$, for each $i \in T$. As the family $\{\pi_i\}_{i \in T}$ is separating, we conclude that $\rho = 0$, a contradiction. 
\end{proof}

\begin{lemma} \label{lm:subhom-justinf}
No separable subhomogeneous \Cs{} is just-infinite.
\end{lemma}

\begin{proof} Let $A$ be a separable just-infinite \Cs. Then $A$ is prime; cf.\ Lemma~\ref{lm:prime}, hence primitive, and so $A$ admits a faithful irreducible representation. Such a representation cannot be finite dimensional, because $A$ is infinite dimensional. Hence $A$ cannot be subhomogeneous; cf.\ Proposition~\ref{prop:subhom}.
\end{proof}

\begin{proposition} \label{prop:char-seq-infty}
Let $A$ be a separable RFD just-infinite   \Cs{} with characteristic sequence $\{k_j\}_{j=1}^\infty$. Then $\lim_{j \to \infty} k_j = \infty$.  
\end{proposition}

\begin{proof} Let $I_1, I_2, \dots$ be the non-zero primitive ideals of $A$, and for each $j$, let $\pi_j$ be an irreducible representation of $A$ whose kernel is $I_j$, such that the dimension of $\pi_j$ is $k_j$. 
We must show that for each $k$, $T_k :=\{j \in \N: k_j \le k\}$ is finite. Suppose that the set $T=T_k$ is infinite. Then the \sh{} $\Psi_T = \bigoplus_{j \in T} \pi_j$ is injective, which implies that $\{\pi_j\}_{j \in T}$ is a separating family of irreducible representations of $A$, each of which having dimension less than or equal to $k$. Then Proposition~\ref{prop:subhom} implies that $A$ is subhomogeneous, but this is impossible by Lemma~\ref{lm:subhom-justinf}.
\end{proof}

\section{Examples of RFD just-infinite AF-algebras} \label{sec:just-inf-AF}

\noindent We construct an example of  a RFD just-infinite  AF-algebra. By Theorem~\ref{thm:types}, its primitive ideal space must be $Y_\infty$.
The existence of a unital AF-algebra whose primitive ideal space is homeomorphic to  $Y_\infty$ follows from Theorem~\ref{thm:BraEll} (Bratteli--Elliott). To conclude that such an AF-algebra is just-infinite, we must also ensure that its non-faithful irreducible representations are finite dimensional; cf.\ Theorem~\ref{thm:types}. This is accomplished by taking a closer look at the construction by Bratteli and Elliott, done in Proposition~\ref{prop:Brattelidiagram} below.

\subsection{Construction of a RFD just-infinite AF-algebra} \label{subsec:A}

Recall that a Bratteli diagram is a graph $(V,E)$, where $V = \bigcup_{n=1}^\infty V_n$ and $E = \bigcup_{n=1}^\infty E_n$ (disjoint unions), all $V_n$ and all $E_n$ are finite sets, and where each edge $e \in E_n$ connects a vertex $v \in V_n$ to a vertex in $w \in V_{n+1}$. In this case, we write $s(e)=v$ and $r(e) = w$, thus giving rise to the \emph{source} and the \emph{range} maps $s,r \colon E \to V$. It was shown by Bratteli, \cite{Bra:AF}, that there is a bijective correspondence between Bratteli diagrams (modulo a natural equivalence class of these) and AF-algebras (modulo Morita equivalence).

An \emph{ideal} in a Bratteli diagram $(V,E)$ is a subset $U \subseteq V$ with the following properties: 
\begin{itemize}
\item for all $e$ in $E$, if $s(e)$ belongs to $U$, then so does $r(e)$,
\item for all $v$ in $V$, if $\{r(e)  \mid  e \in s^{-1}(v)\}$ is contained in $U$, then $v$ belongs to $U$.
\end{itemize}
The ideal lattice of an AF-algebra associated with a given Bratteli diagram is isomorphic to the ideal lattice of the Bratteli diagram, see \cite{Eff:AF} or \cite{Dav:C*-ex}. The following proposition is contained in \cite{BraEll:AF}:

\begin{proposition}[Bratteli--Elliott] \label{prop:Brattelidiagram}
Let $X$ be a second countable, compact, totally disconnected T$_0$-space. Let $\cG_1, \cG_2,  \dots$ be finite families of  compact-open subsets of $X$ such that:
\begin{enumerate}
\item $X = \bigcup_{G \in \cG_n}  G$,  for each $n \ge 1$,
\item for each $n \ge 1$, $\cG_{n+1}$ is a refinement of $\cG_n$, i.e., each set in $\cG_{n+1}$ is contained in a set in $\cG_n$, and each set in $\cG_n$ is the union of sets from $\cG_{n+1}$,
\item $\displaystyle{\bigcup_{n=1}^\infty \cG_n}$ is a basis for the topology on $X$.
\end{enumerate}
Consider the Bratteli diagram for which the vertices at level $n$ are the sets in $\cG_n$, and where there is one edge from $G \in \cG_n$ to $G' \in \cG_{n+1}$ if $G'\subseteq G$, and none otherwise. 
Then there is a one-to-one correspondence between open subsets of $X$ and ideals of the Bratteli diagram, given as follows: the ideal in the Bratteli diagram associated with an open subset $U$ of $X$ consists of all vertices $G \in \bigcup_{n=1}^\infty \cG_n$  for which $G \subseteq U$. 

If, in addition, $X$ is a spectral space, and if $A$ is an AF-algebra associated with the Bratteli diagram constructed above, then $\Prim(A)$ is homeomorphic to $X$. 
\end{proposition}

\noindent In the following, we construct a sequence $\cG_1, \cG_2, \cG_3, \dots$ of finite families of compact-open subsets of $X$ satisfying the conditions of Proposition~\ref{prop:Brattelidiagram} in the case where $X = Y_\infty$. 
Recall that $Y_\infty = \{0\} \cup \N$, that the open subsets of $Y_\infty$ are $\emptyset$, $Y_\infty$, and all co-finite subsets of $\N$, and that every subset of $Y_\infty$ is compact. For all $n \ge 1$, set
$$F_{n,k} = \{1,2, \dots, n\} \setminus \{k\}, \qquad G_{n,k} = Y_\infty \setminus F_{n,k},$$
for $1 \le k \le n$, and let
$$\cG_n = \{G_{n,1}, G_{n,2}, \dots, G_{n,n}\}.$$
Observe that each $G_{n,k}$ is open and (automatically) compact. Moreover, the sets $\{\cG_n\}_{n=1}^\infty$  satisfy conditions (i), (ii) and (iii)  in Proposition \ref{prop:Brattelidiagram}. Furthermore, for $1 \le k \le n$,  
$$G_{n+1,k} \subseteq G_{n,k}, \qquad  G_{n+1,n+1} \subseteq G_{n,k}.$$
No other inclusions between sets in $\cG_{n+1}$ and sets in $\cG_n$ hold. Therefore, the Bratteli diagram associated with this sequence of compact-open subsets of $Y_\infty$ as in Proposition \ref{prop:Brattelidiagram} is:

\begin{center}
\begin{equation} \label{eq:Bratteli}
\begin{split}
\xymatrix@R-1.3pc{\bullet \ar@{-}[dd] \ar@{-}[ddr]& & & & \\ & & & &\\
\bullet \ar@{-}[dd] \ar@{-}[ddrr]& \bullet  \ar@{-}[dd] \ar@{-}[ddr]&& &\\  & & & &\\
\bullet  \ar@{-}[dd] \ar@{-}[ddrrr]& \bullet  \ar@{-}[dd] \ar@{-}[ddrr]& \bullet  \ar@{-}[dd] \ar@{-}[ddr]& & \\  & & & &\\
\bullet  \ar@{-}[dd] \ar@{-}[ddrrrr] & \bullet  \ar@{-}[dd] \ar@{-}[ddrrr] & \bullet  \ar@{-}[dd] \ar@{-}[ddrr]& \bullet  \ar@{-}[dd] \ar@{-}[ddr]& \\  & & & &\\
\bullet  & \bullet   & \bullet   & \bullet  & \bullet  \\  \vdots & \vdots & \vdots & \vdots & \vdots
} 
\hspace{1.5cm}
\xymatrix@R-0.52pc@C-2.2pc{\C & && && &&& \\
\C  & \oplus&\C &&& &&& \\ 
\C  & \oplus & \C  &\oplus & M_2(\C) & && &\\ 
\C  & \oplus & \C  & \oplus & M_2(\C) & \oplus  & M_4(\C)& & \\ 
\C  & \oplus &\C & \oplus  & M_2(\C) & \oplus  & M_4(\C)& \oplus  & M_8(\C)\\ 
}
\end{split}
\end{equation}
\end{center}

\noindent The sequence of finite dimen\-sio\-nal \Cs s on the right-hand side, equipped with unital connecting mappings given by the Bratteli diagram, defines a unital AF-algebra $A$, associated with the Bratteli diagram. The one-to-one correspondence between (non-empty) open subsets $G \subseteq Y_\infty = \{0\} \cup \N$ and ideals $U(G)$ of the Bratteli diagram above is given as follows:
$$U(G) = \{G_{n,k} \mid G_{n,k} \subseteq G\} = \{ G_{n,k} \mid k \in G, \, n \ge \max Y_\infty \setminus G\}.$$
E.g., $U(Y_\infty \setminus \{1,3\}) = \{G_{n,k} \mid n \ge 3, \, k \ne 1,3\}$ and
$U(Y_\infty\setminus \{j\}) = \{G_{n,k} \mid n \ge j, \, k \ne j\}$, $j \ge 1$.

The quotient of the AF-algebra $A$ by the ideal in $A$ corresponding to $U(G)$ is given by the Bratteli diagram that arises by removing $U(G)$ from the original diagram. The two pictures below show the ideal $U(G)$ 
%color
(in blue) % choose this for colored
%(dotted lines and open vertices): %choose this for black and white
and the Bratteli diagram of the quotient 
(in red) % choose this for colored
%(bold lines and filled vertices) %choose this for blakc and white
in the cases where $G= Y_\infty \setminus \{2\}$, respectively, $G = Y_\infty \setminus \{1,3\}$:

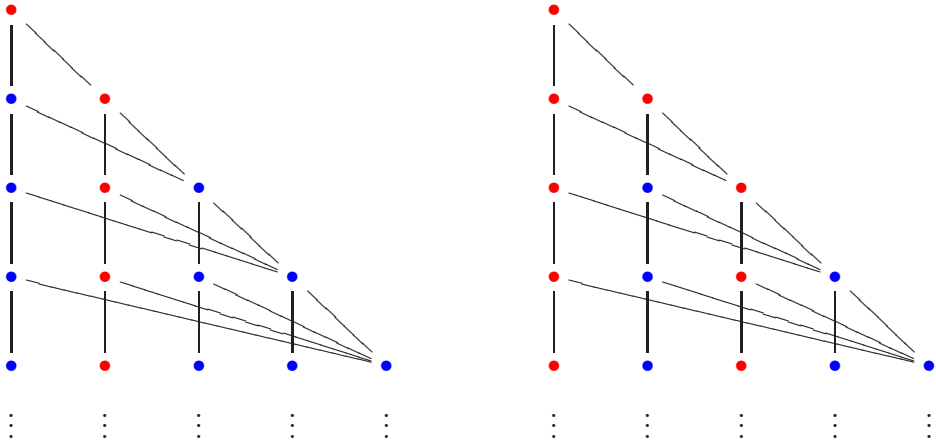
$$
\xymatrix@R-1.3pc{\rbullet \bluearrow[dd] \redarrow[ddr]& & & & & \\ &&&& \\
\bbullet \bluearrow[dd] \bluearrow[ddrr]& \rbullet  \redarrow[dd] \bluearrow[ddr]&& &&\\ &&&& \\
\bbullet  \bluearrow[dd] \bluearrow[ddrrr]& \rbullet  \redarrow[dd] \bluearrow[ddrr]& \bbullet  \bluearrow[dd] \bluearrow[ddr]& &&\\ &&&& \\
\bbullet  \bluearrow[dd] \bluearrow[ddrrrr] & \rbullet  \redarrow[dd] \bluearrow[ddrrr] & \bbullet  \bluearrow[dd] \bluearrow[ddrr]& \bbullet  \bluearrow[dd] \bluearrow[ddr]&&\\ &&&& \\
\bbullet  & \rbullet   & \bbullet   & \bbullet  & \bbullet& \\
\vdots & \vdots & \vdots & \vdots & \vdots & 
} 
\qquad 
\xymatrix@R-1.3pc{\rbullet  \redarrow[dd] \redarrow[ddr]& & & & & \\ &&&& \\
\rbullet \redarrow[dd] \redarrow[ddrr]& \rbullet  \bluearrow[dd] \redarrow[ddr]&& &&\\ &&&& \\
\rbullet  \redarrow[dd] \bluearrow[ddrrr]& \bbullet  \bluearrow[dd] \bluearrow[ddrr]& \rbullet  \redarrow[dd] \bluearrow[ddr]& &&\\ &&&& \\
\rbullet  \redarrow[dd] \bluearrow[ddrrrr] & \bbullet  \bluearrow[dd] \bluearrow[ddrrr] & \rbullet  \redarrow[dd] \bluearrow[ddrr]& \bbullet  \bluearrow[dd] \bluearrow[ddr]&&\\ &&&& \\
\rbullet  & \bbullet   & \rbullet   & \bbullet  & \bbullet& \\
\vdots & \vdots & \vdots & \vdots & \vdots & 
} 
$$
The quotient of $A$ by the ideal in $A$ corresponding to $U(G)$, with $G = Y_\infty \setminus \{2\}$, is the AF-algebra associated to the red part of the Bratteli diagram, which is $\C$. The quotient of $A$ in the case where $G = Y_\infty \setminus \{1,3\}$ is similarly seen to be $\C \oplus M_2(\C)$. 

By construction, and by Proposition~\ref{prop:Brattelidiagram}, we have $\Prim(A) \cong Y_\infty$. In more detail, we have $\Prim(A) = \{0\} \cup \{I_1,I_2, I_3, \dots\}$, where $I_j$ is the primitive ideal in $A$ corresponding to the ideal $U(Y_\infty \setminus \{j\})$ of the Bratteli diagram. Arguing as in the two examples above, we see that $A/I_j \cong M_{k(j)}(\C)$, where $k(1) = k(2) = 1$ and $k(j) = 2^{j-1}$, for $j \ge 2$. Hence $A/I$ is finite dimensional, for each non-zero primitive ideal $I$ of $A$. It now follows from Theorem~\ref{thm:types} that $A$ is just-infinite and RFD, as desired. The characteristic sequence of $A$ is precisely the sequence $\{k(j)\}_{j=1}^\infty$ defined above.

One can modify the Bratteli diagram in various ways to construct new RFD just-infinite AF-algebras with other characteristic sequences. For example, one can delete the first $n-1$ rows and let row $n$ correspond to an arbitrary finite dimensional \Cs{} with $n$ summands. (The remaining finite dimensional \Cs s are then determined by the one chosen and by the Bratteli diagram.) One is also allowed to change the multiplicity of the edges connecting the vertex at position $(n,k)$, $1 \le k \le n$, to the vertex at position $(n+1,n+1)$.  In these examples, the characteristic sequences all grow exponentially. By Proposition~\ref{prop:char-seq-infty}, we know that they must tend to infinity. This leaves open the following:

\begin{question} What are the possible characteristic sequences $\{k_j\}_{j=1}^\infty$ of RFD just-infinite \Cs s? 
\end{question}

\subsection{The dimension group}

We compute the dimension group $(K_0(A), K_0(A)^+, [1])$ of the just-infinite AF-algebra $A$ constructed above (associated with the Bratteli diagram \eqref{eq:Bratteli}). 

Recall that the dimension group, $(H,H^+,v)$, associated with the Bratteli diagram \eqref{eq:Bratteli} is the inductive limit of the ordered abelian groups
\begin{eqnarray*}
&\xymatrix{ \Z \ar[r]^-{\alpha_1} & \Z^2 \ar[r]^-{\alpha_2} &  \Z^3 \ar[r]^-{\alpha_3} & \cdots,}\\
&\alpha_n(x_1,x_2, \dots,x_n) = (x_1,x_2, \dots, x_n,x_1+x_2+ \cdots + x_n), \qquad (x_1, \dots, x_n) \in \Z^n,
\end{eqnarray*}
where $v \in H^+$ is the image of $1$ in the first copy of $\Z$. It follows from standard theory of AF-algebras that $(K_0(A), K_0(A)^+, [1])$ is isomorphic to $(H,H^+,v)$. We proceed to identify the latter more explicitly.

Let $\prod_{j \in \N} \Z$ denote the (uncountable) group of all sequences $x=\{x_j\}_{j=1}^\infty$ of integers, equipped with the usual order: $x \ge 0$ if and only if $x_j \ge 0$, for all $j \ge 1$. Let $G$ be the countable subgroup of $\prod_{j \in \N} \Z$ consisting of those sequences $\{x_j\}_{j=1}^\infty$ for which the identity $x_{j+1} = x_1+x_2 + \cdots + x_j$ holds eventually, and equip $G$ with the order inherited from $\prod_{j \in \N} \Z$.
Set $u = (1,1,2,4,8, \cdots)$. We show below that $(H,H^+,v) \cong (G,G^+,u)$. In conclusion, 
$$(K_0(A),K_0(A)^+, [1])  \cong (H,H^+,v) \cong (G,G^+,u).$$

For this, define $\rho_n \colon \Z^n \to G$ by
$$\rho_n(x_1,x_2, \dots, x_n) = (x_1, x_2, \dots, x_n, x_{n+1}, x_{n+2}, \dots),$$
where $x_{j+1} = x_1+x_2 + \cdots + x_j$, for all $j \ge n$. Then $\rho_{n+1} \circ \alpha_n = \rho_n$, for all $n$, and each $\rho_n$ is positive. It follows that the $\rho_n$'s extend to a  positive group homomorphism $\rho \colon H \to G$. Each $\rho_n$ is injective, so $\rho$ is injective. 

To complete the proof that $\rho$ is an order isomorphism, we show that $\rho(H^+) = G^+$. Take $x=\{x_j\}_{j=1}^\infty  \in G^+$, and let $n \ge 1$ be such that $x_{j+1} = x_1+x_2 + \cdots + x_j$, for all $j \ge n$. Then 
\begin{eqnarray*}
x &=& \rho_n(x_1,x_2, \dots, x_n) \\ &=&x_1 \rho_n(e_1^{(n)}) +  x_2 \rho_n(e_2^{(n)})  + \cdots +  x_n \rho_n(e_n^{(n)})\\
&=& x_1 \rho(f_1^{(n)})+x_2\rho(f_2^{(n)}) + \cdots + x_n \rho(f_n^{(n)}),
\end{eqnarray*}
where  $e_1^{(n)}, e_2^{(n)}, \dots, e_n^{(n)}$ is the standard basis for $(\Z^n)^+ \subseteq \Z^n$, and  $f_1^{(n)}, f_2^{(n)}, \dots, f_n^{(n)}$  are the corresponding images in $H^+  \subseteq H$. This shows that $x \in \rho(H^+)$. Finally, $\rho(v) = \rho_1(1) = u$, as wanted.

\bigskip\noindent
Unital AF-algebras are completely classified by their ordered $K_0$-group, together with the position of the class of the unit. It is therefore an interesting question to classify, or characterize, those dimension groups which are the $K_0$-group of a RFD just-infinite AF-algebra.

In the light of the computation above, one may first wish to consider those dimension groups $G$ which are (ordered) subgroups of $\prod_{j=1}^\infty \Z$. In addition, one should assume that $G$ is a subdirect product of  $\prod_{j=1}^\infty \Z$, in the sense that $\varphi_F(G) = \prod_{j \in F} \Z$, for each finite subset $F$ of $\N$, where $\varphi_F$ is the canonical projection of $\prod_{j =1}^\infty \Z$ onto $\prod_{j \in F} \Z$. The dimension group considered above has this property.

\subsection{A strictly RFD \Cs{} which is not just-infinite} \label{subsec:SRFDvsJI}

\noindent It was shown in Corollary~\ref{cor:strictRFD} that all RFD just-infinite \Cs s are strictly RFD. We show here that the converse does not hold, by constructing an example of a unital AF-algebra which is strictly RFD and not just-infinite.

Let us first describe the example at the level of its primitive ideal space. Let $X$ be the disjoint union of  two copies of $Y_\infty$, i.e., $X= X_1 \amalg X_2$, where $X_1 = X_2 = Y_\infty$. Equip $X$ with the following topology: A non-empty subset $U$ of $X$ is open if and only if $U \cap X_1$ is non-empty and open, and $U \cap X_2$ is open. That this indeed defines a topology on $X$ follows from the fact that the intersection of any two non-empty open subsets of $X_1$ is non-empty, or, equivalently, that the set $X_1$ is prime.

Observe that $X_2$ is an infinite closed subset of $X$. Hence $X_2$ is a non-dense infinite subset of $X$. This shows that $X$ cannot be the primitive ideal space of a just-infinite \Cs; cf.\ Lemma~\ref{lm:Prim-ji}. The set $X_1$, on the other hand, is an open and dense subset of $X$, and each infinite subset of $X_1$ is dense in $X_1$, and therefore also dense in $X$. 

The space $X$ is the primitive ideal space of the unital AF-algebra $B$ whose Bratteli diagram is given as follows 
%color
(ignoring at first the coloring of the vertices and edges): %choose this for colored
%(ignoring at first the shading of the edges and vertices): %choose this for black and white

$$
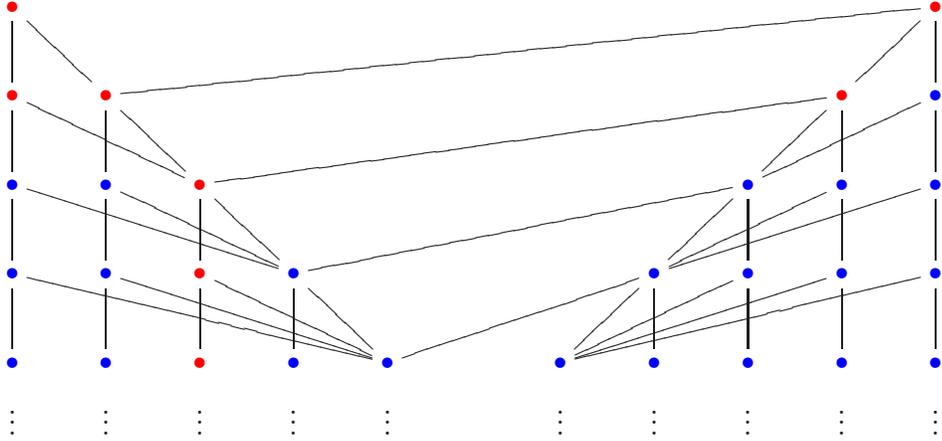
\xymatrix@R-1.3pc{\rbullet \redarrow[dd] \redarrow[ddr]& & & & &  &&&&& \rbullet \bluearrow[dd] \redarrow[ddl] \redarrow[ddlllllllll]\\ 
&&&&&&&&&&\\
\rbullet \bluearrow[dd] \redarrow[ddrr]& \rbullet  \bluearrow[dd] \redarrow[ddr]&& && &&&& \rbullet \bluearrow[dd]  \bluearrow[ddl] \redarrow[ddlllllll]& \bbullet \bluearrow[dd] \bluearrow[ddll]
\\ &&&&&&&&&&\\
\bbullet  \bluearrow[dd] \bluearrow[ddrrr]& \bbullet  \bluearrow[dd] \bluearrow[ddrr]& \rbullet  \redarrow[dd] 
\bluearrow[ddr]& && & & & \bbullet \bluearrow\bluearrow[dd] \bluearrow[ddl] \bluearrow[ddlllll] & \bbullet \bluearrow[dd] \bluearrow[ddll] & \bbullet \bluearrow[dd] \bluearrow[ddlll] \\ 
&&&&&&&&&& \\
\bbullet  \bluearrow[dd] \bluearrow[ddrrrr] & \bbullet  \bluearrow[dd] \bluearrow[ddrrr] & \rbullet  \redarrow[dd] \bluearrow[ddrr]& \bbullet  \bluearrow[dd] \bluearrow[ddr]&&&& \bbullet \bluearrow[dd] \bluearrow[ddl] \bluearrow[ddlll] & \bbullet \bluearrow[dd] \bluearrow[ddll] & \bbullet \bluearrow[dd] \bluearrow[ddlll] & \bbullet \bluearrow[dd] \bluearrow[ddllll]
\\ &&&&&&&&&& \\
\bbullet  & \bbullet   & \rbullet   & \bbullet  & \bbullet& &\bbullet  & \bbullet   & \bbullet   & \bbullet  & \bbullet\\
\vdots & \vdots & \vdots & \vdots & \vdots &  & \vdots & \vdots & \vdots & \vdots & \vdots
} 
$$

The left-hand half of this Bratteli diagram is an essential ideal in the Bratteli diagram\footnote{An ideal $U$ in a Bratteli diagram is said to be \emph{essential}, if $U \cap V \ne \emptyset$, for all non-empty ideals $V$.} and therefore corresponds to an essential ideal $I$ of the AF-algebra $B$. The right-hand half is the Bratteli diagram of the quotient $B/I$. Hence $B/I$ is equal to the RFD just-infinite  AF-algebra $A$ described in Section~\ref{subsec:A}, and $I$ is Morita equivalent to $A$. Hence $B$ cannot be just-infinite. For each $k \ge 1$, let $U_k$ be largest ideal of the Bratteli diagram which does not contain any vertex from the $k$th column of the left-hand half of the Bratteli diagram. Furthermore, let $I_k$ be the ideal of $B$ corresponding to the ideal $U_k$. 

To illustrate this definition, in the diagram above, the ideal $U_3$ is 
%color
marked in blue % choose this for colored
%marked with dotted lines and open vertices %choose this for black and white
and the Bratteli diagram of the quotient $B/I_3$ is
marked in red % choose this for colored
%marked with bold lines and filled vertices %choose this for blakc and white
The quotient $B/I_3$ is seen to be isomorphic to $M_4(\C)$. 

In general, for each $k \ge 1$, we see that $B/I_k$ is a full matrix algebra, and (hence) that each $I_k$ is a primitive ideal.  Moreover, $\bigcap_{k \in T} I_k = 0$, for each infinite subset $T$ of $\N$. (To see this, observe that $U_k$ contains no vertices from the top $k-1$ rows of the left-hand half of the Bratteli diagram, or from the top $k-2$ rows of the right-hand half. Hence $\bigcap_{k \in T} U_k = \emptyset$, for each infinite subset $T$ of $\N$.)

This shows that $B$ is a strictly RFD AF-algebra which is not just-infinite.

\section{Subalgebras and superalgebras} \label{sec:subalgebras}

\noindent In this section, which is addressed to specialists in \Cs s, we investigate when subalgebras and superalgebras of just-infinite \Cs s are again just-infinite, and we also show that not all RFD just-infinite \Cs s are nuclear, or even exact. The third named author thanks Jose Carrion for his suggestion to use Theorem~\ref{thm:Dad} below of Dadarlat to conclude that there are non-nuclear, and even non-exact, RFD just-infinite \Cs s. 

Recall that a \Cs{} $A$ has \emph{real rank zero} if each self-adjoint element in $A$ is the norm limit of self-adjoint elements in $A$ with finite spectra. A commutative \Cs{} $C(X)$ has real rank zero if and only if $X$ is totally disconnected (or, equivalently, $\mathrm{dim}(X) = 0$). Real rank zero is therefore viewed as a non-commutative analog of being zero-dimensional. A \Cs{} has real rank zero if it has ``sufficiently many projections''. Each closed two-sided ideal of a \Cs{} of real rank zero again has real rank zero and, as a consequence, is generated by its projections.

We denote by $\Ideal(A)$ the lattice of closed two-sided ideals in $A$. If $B$ is a sub-\Cs{} of $A$, then there is a natural map $\Phi \colon \Ideal(A) \to \Ideal(B)$, given by $\Phi(I) = I \cap B$. The map $\Phi$ is, in general, neither injective nor surjective, but it is both in the special situation of the lemma below. We use the symbol $p \sim_A q$ to denote that $p$ and $q$ are Murray-von Neumann equivalent projections, relatively to the \Cs{} $A$.

\begin{lemma} \label{lm:A-B-ideals} Let $B \subseteq A$ be unital \Cs s of real rank zero, and suppose that
there is a \sh{} $\kappa \colon A \to B$ such that $\kappa(p) \sim_A p$,  for all projections $p \in A$, and $\kappa(q) \sim_B q$, for all projections $q \in B$. Then the map $\Phi \colon \Ideal(A) \to \Ideal(B)$  is a lattice isomorphism. 
\end{lemma}

\begin{proof} We first show that $\Phi$ is injective. Let $I \ne I' \in \Ideal(A)$ be given. Since $A$ has real rank zero, and ideals in $A$ are generated by their projections, there exists a projection $p \in I$ such that $p \notin I'$ (or vice versa). Set $q = \kappa(p) \sim p$. Then $q \in I \cap B = \Phi(I)$, but $q \notin I' \cap B = \Phi(I')$. Hence $\Phi(I) \ne \Phi(I')$.

Let now $J \in \Ideal(B)$ be given, and let $I = \overline{AJA}$ be the closed two-sided ideal in $A$ generated by $J$. Then, clearly, $J \subseteq I \cap B = \Phi(I)$. To see that $\Phi(I) \subseteq J$, it suffices to show that each projection $q$  in $\Phi(I)$ belongs to $J$. Being a projection in $I$, $q$ belongs to the algebraic two-sided ideal in $A$ generated by $J$, so $q = \sum_{j=1}^n a_j x_j b_j$ for some $a_j,b_j \in A$ and $x_j \in J$. The conditions on $\kappa$, together with the fact that $B$ is a \Cs{} of real rank zero, imply that $\kappa$ maps $J$ into itself, so 
$$q \sim_B \kappa(q) = \sum_{j=1}^n \kappa(a_j) \kappa(x_j) \kappa(b_j) \in J.$$
This shows that $q$ belongs to $J$, as desired.
\end{proof}

\begin{lemma} \label{lm:ji-subalgebra} Let $A$ be a unital separable RFD just-infinite \Cs{} of real rank zero, and let  $\{\pi_n\}_{n=1}^\infty$ be an exhausting\footnote{See Remark~\ref{rem:add-structure}.} sequence of pairwise inequivalent non-faithful irreducible representations of $A$. 
\begin{enumerate}
\item Suppose that  $B$ is a unital sub-\Cs{} of  $A$ such that the map $\Phi \colon \Ideal(A) \to \Ideal(B)$ is an isomorphism, and such that each projection in $A$ is equivalent to a projection in $B$. It follows that $B$ is just-infinite and RFD, that $\{\pi_n|_{B}\}_{n=1}^\infty$ is an exhausting sequence of pairwise inequivalent non-faithful irreducible representations of $B$, and that $\pi_n(B) = \pi_n(A)$, for all $n$. In particular, $A$ and $B$ have the same characteristic sequence. 
\item Suppose that  $C$ is a unital \Cs{} of real rank zero which contains $A$ and is asymptotically homotopy equivalent\footnote{This means that there exists an asymptotic morphism $C \to A$, so that the asympotic morphism $C \to C$ (obtained by composing it with the inclusion mapping $A \to C$) is homotopic to the identity on $C$ in the category of asymptotic morphism. See also \cite{Dad:non-exact}.}
to $A$. Suppose also that $\Phi \colon \Ideal(C) \to \Ideal(A)$ is an isomorphism.\footnote{In fact, the assumptions on $A$ and $C$  imply that $\Phi$ is an isomorphism. This can be shown along the same lines as the proof of Lemma~\ref{lm:A-B-ideals}.} It follows that $C$ is just-infinite and RFD with an exhausting sequence $\{\nu_n\}_{n=1}^\infty$ of pairwise  inequivalent non-faithful irreducible representations for which $\Ker(\nu_n|_A) = \Ker(\pi_n)$ and $\nu_n(C) \cong \pi_n(A)$, for all $n$. In particular, $A$ and $C$ have the same characteristic sequence. 
\end{enumerate}
\end{lemma}

\begin{proof} (i). The lattice isomorphism $\Phi \colon \Ideal(A) \to \Ideal(B)$ restricts to a homeomorphism $\Prim(A) \to \Prim(B)$, and so $\Prim(B)$ is homeomorphic to $\Prim(A)$, which again is homeomorphic to $Y_\infty$.  Moreover, 
 $$\Prim(B) \setminus \{0\} = \{ \Ker(\pi_n) \cap B \mid n \in \N\}= \{ \Ker(\pi_n|_B) \mid n \in \N\}.$$

Let $I$ be a non-zero primitive ideal of $B$. Then $I = \Ker(\pi_n|_B)$, for some $n$, and $B/I$ is isomorphic to $\pi_n(B)$, which is a subalgebra of the finite dimensional \Cs{} $\pi_n(A)$, so $B/I$ is finite dimensional. It now follows from Theorem~\ref{thm:types} that $B$ is just-infinite.

Let us also show that $\pi_n(B) = \pi_n(A)$, for all $n$. Since $\pi_n(B) \subseteq \pi_n(A)$ and both \Cs s are full matrix algebras, it  suffices to show that $\pi_n(B)$ contains a minimal projection in $\pi_n(A)$. Let $e \in \pi_n(A)$ be such a projection and lift it to a projection $p \in A$ (which is possible because $A$ is assumed to have real rank zero). Find a projection $q \in B$ which is equivalent to $p$. Then $\pi_n(q)$ is equivalent to $e$, which implies that $\pi_n(q)$ itself is a minimal projection in $\pi_n(A)$.

(ii). As in (i), the given lattice isomorphism $\Phi \colon \Ideal(C) \to \Ideal(A)$ restricts to a  homeomorphism $\Prim(C) \to \Prim(A)$, so $\Prim(C)$ is homeomorphic to $Y_\infty$. 

Given $n \ge 1$, let $J_n \in \Prim(C)$ be such that $\Phi(J_n) = \Ker(\pi_n)$. Since $\Phi$ is an isomorphism, each non-zero primitive ideal in $C$ is of this form. Identify $\pi_n(A)$ with $M_k(\C)$, for some positive integer $k$. Find an irreducible representation $\nu_n \colon C \to B(H)$ on some Hilbert space $H$, with $\Ker(\nu_n) = J_n$. Then $\Ker(\pi_n) = \Phi(J_n) =J_n \cap A = \Ker(\nu_n|_A)$. Let $\iota \colon M_k(\C) \to B(H)$ be the inclusion mapping making the following diagram commutative:
$$\xymatrix{A \ar@{^{(}->}[r]  \ar[d]_-{\pi_n} & C  \ar[d]^-{\nu_n} \\
M_k(\C) \ar@{^{(}->}_\iota[r] & B(H)}$$
We show that $\dim(H) = k$, which by Theorem \ref{thm:types}, will imply that $C$ is just-infinite. It will also imply that  $\iota$ is an isomorphism, and that $\pi_n(A) \cong \nu_n(A) = \nu_n(C) = B(H)$.

It is clear that $\dim(H) \ge k$. Suppose that $\dim(H) > k$. Then we can find pairwise orthogonal non-zero projections $f_1,f_2, \dots, f_{k+1}$ in $\nu_n(C)$. (Indeed, $\nu_n(C)$ acts irreducibly on $H$, so if $\dim(H)$ is finite, then $\nu_n(C) = B(H)$. If $\dim(H)$ is infinite, then $\nu_n(C)$ is infinite dimensional and of real rank zero. In either case, one can find the desired projections.) 
Since $C$ has real rank zero, we can lift the projections $f_1, f_2, \dots, f_{k+1}$ to mutually orthogonal projections $p_1,p_2, \dots, p_{k+1}$ in $C$. Applying the asymptotic homomorphism $C \to A$ to the projections $p_1,p_2, \dots, p_{k+1}$,  and using that $\C^{k+1}$ is semiprojective (see \cite{Bla:semiprojective}), we obtain mutually orthogonal projections $q_1,q_2, \dots, q_{k+1}$ in $A$. Since the asymptotic homomorphism $C \to A$  composed with the inclusion mapping $A \to C$ is homotopic to the identity mapping on $C$, we further get  that $q_j$ is equivalent to $p_j$, for each $j$. In particular, $(\iota \circ \pi_n)(q_j) = \nu_n(q_j) \sim  \nu_n(p_j) = f_j$, for each $j$, so $\pi_n(q_j)$ is non-zero. But $M_k(\C)$ does not contain $k+1$ mutually orthogonal non-zero projections. This proves that $\dim(H) = k$. 
\end{proof}

\noindent We shall combine Lemma~\ref{lm:ji-subalgebra} with the following results due to Dadarlat:

\begin{theorem}[{Dadarlat, \cite[Theorem 11 and Proposition 9]{Dad:non-exact}}]  \label{thm:Dad}
Let $A$ be a unital AF-algebra not of type I. Then:
\begin{enumerate}
\item $A$ contains a unital non-nuclear sub-\Cs{} $B$ of real rank zero and stable rank one, for which there exists a unital $^*$-monomorphism $\kappa \colon A \to B$ such that $\iota \circ \kappa$ is homotopic to $\Id_A$ and $\kappa \circ \iota$ is asymptotically homotopic to $\Id_B$, where $\iota$ is the inclusion mapping $B \to A$. Moreover, $\Phi \colon \Ideal(A) \to \Ideal(B)$ is an isomorphism of lattices.
\item $A$ is contained in a unital separable non-exact \Cs{} $C$ of real rank zero and stable rank one, which is asymptotically homotopy equivalent to $A$, and for which $\Phi \colon \Ideal(C) \to \Ideal(A)$ is an isomorphism of lattices.
\end{enumerate}
\end{theorem}

\noindent The statements (i) and (ii) that $\Phi$ is an isomorphism between the ideal lattices of $A$ and $B$, respectively, of $C$ and $A$, are included in the quoted results of Dadarlat, and it also follows from Lemma~\ref{lm:A-B-ideals}  in the situation considered in (i).

To apply Theorem~\ref{thm:Dad}, we need the following:

\begin{lemma} \label{lm:antiliminal} A separable just-infinite \Cs{} is of type I if and only if \emph{either} it is isomorphic to $\cK$, the compact operators on a separable infinite dimensional Hilbert space, \emph{or} it is an essential extension of $\cK$ by a finite dimensional \Cs. In the former case, $A$ is just-infinite of type ($\alpha$), and in the latter case $A$ is just-infinite of type ($\beta$); cf.\ Theorem~\ref{thm:types}.

In particular, no just-infinite \Cs{} of type ($\gamma$), i.e., RFD, is of type I.
\end{lemma}

\begin{proof} Let $A$ be a separable just-infinite \Cs{} of type I. By Lemma~\ref{lm:prime}, $A$ is prime, and hence primitive, so we can find a faithful irreducible representation $\pi$ of $A$ on a separable, necessarily infinite dimensional, Hilbert space $H$. Since $A$ is a \Cs{} of type I, the algebra $\cK$ of compact operators  on $H$ is contained in the image of $\pi$. Hence  $I = \pi^{-1}(\cK)$ is a non-zero closed two-sided ideal in $A$,  which is isomorphic to $\cK$. As $A$ is just-infinite, either $I=A$, or $A/I$ is finite dimensional. 
\end{proof}

\begin{corollary} \label{cor:Dad-subalgebra}  Let $A$ be a unital separable RFD just-infinite AF-algebra, and let $\{\pi_n\}_{n=1}^\infty$ be an exhausting sequence of pairwise  inequivalent non-faithful irreducible representations of $A$.
It follows that $A$ contains a unital non-nuclear RFD just-infinite sub-\Cs{} $B$ of real rank zero such that $\{\pi_n|_B\}_{n=1}^\infty$ is an exhausting sequence of pairwise  inequivalent non-faithful irreducible representations of $B$, and $\pi_n(B) = \pi_n(A)$, for all $n$. 
\end{corollary}

\begin{proof} By Lemma \ref{lm:antiliminal}, we can now apply Theorem~\ref{thm:Dad} (i) to find a sub-\Cs{} $B$ of $A$ with the properties listed therein. Each projection $p \in A$ is equivalent to a projection in $B$. Indeed, set $q = \kappa(p) \in B$. Then $q = (\iota \circ \kappa)(p)$ is homotopic (and hence equivalent) to $p$. The desired conclusion now follows from Lemma~\ref{lm:ji-subalgebra}~(i).
\end{proof}

\begin{corollary} \label{cor:Dad-subalgebra-2}  Let $A$ be a unital separable RFD just-infinite AF-algebra. Then $A$ is contained in a separable non-exact unital  RFD just-infinite \Cs{} $C$ of real rank zero, equipped with an exhausting sequence of pairwise  inequivalent non-faithful irreducible representations $\{\nu_n\}_{n=1}^\infty$, such that their restrictions to $A$  form an 
exhausting sequence of pairwise inequivalent non-faithful irreducible representations for $A$, and  $\nu_n(A) = \nu_n(C)$, for all $n$. 
\end{corollary}

\begin{proof} This follows immediately from Lemma \ref{lm:antiliminal}, Theorem~\ref{thm:Dad}~(ii) and Lemma~\ref{lm:ji-subalgebra}~(ii).
\end{proof}

\noindent The two above corollaries, in combination with the existence of a RFD just-infinite AF-algebra (see Section~\ref{sec:just-inf-AF}), now yield the following:

\begin{corollary} \label{cor:non-nuclear} There exist non-nuclear exact RFD just-infinite \Cs s, and there also exist non-exact RFD just-infinite \Cs s. 
\end{corollary}

\noindent It is shown in \cite[Theorem 4.3]{PerRor:AF} that each unital \Cs{} $A$ of real rank zero contains a unital AF-algebra $B$ such that each projection in $A$ is equivalent to a projection in $B$, and such that $\Phi \colon \Ideal(A) \to \Ideal(B)$ is an isomorphism. Together with Lemma~\ref{lm:ji-subalgebra} (i), this proves the following:

\begin{proposition} \label{prop:AF-subalg}  Let $A$ be a unital separable RFD just-infinite \Cs{} of real rank zero, and let $\{\pi_n\}_{n=1}^\infty$ be an exhausting sequence of pairwise inequivalent non-faithful irreducible representations of $A$.
It follows that $A$ contains a unital RFD just-infinite AF-sub-\Cs{} $B$
such that $\{\pi_n|_B\}_{n=1}^\infty$ is an exhausting sequence of pairwise  inequivalent non-faithful irreducible representations of $B$, and $\pi_n(B) = \pi_n(A)$, for all $n$. 
\end{proposition}

\noindent By combining Corollary~\ref{cor:Dad-subalgebra} with Proposition~\ref{prop:AF-subalg}, one obtains the following fact: Suppose that $A$ is a unital separable RFD just-infinite \Cs{} of real rank zero, and  $\{\pi_n\}_{n=1}^\infty$ is an exhausting sequence of pairwise inequivalent non-faithful irreducible representations. Then there is a strictly decreasing sequence $A \supset A_1 \supset A_2 \supset A_3 \supset \cdots$ of unital sub-\Cs s $A_k$ of $A$ such that each $A_k$ is a RFD just-infinite \Cs, and $\pi_n(A_k) = \pi_n(A)$, for all $k$ and $n$. (In fact, every other \Cs{} in the sequence $\{A_k\}$ can be taken to be an AF-algebra and the remaining ones to be non-nuclear.)

In particular, a unital separable RFD just-infinite \Cs{} of real rank zero can never be \emph{minimal} in the sense that it contains no proper RFD just-infinite sub-\Cs s.

\section{Just-infiniteness of group \Cs s} \label{sec:ji-groups}

\noindent We discuss in this section when \Cs s associated with groups are just-infinite. 

The group algebra $\C[G]$ of a group $G$ is in a natural way a $^*$-algebra in such a way that each group element $g \in G$ becomes a unitary in $\C[G]$, and it can be completed to become a \Cs, usually in many ways. The \emph{universal \Cs{}} of $G$, denoted by $C^*(G)$, is the completion of $\C[G]$ with respect to the maximal $C^*$-norm on $\C[G]$. Each unitary representation $\pi$ of the group $G$ on a Hilbert space gives rise to unital $^*$-representations (again denoted by $\pi$) of the $^*$-algebras $\C[G]$ and $C^*(G)$ on the same Hilbert space. Respectively, each unital $^*$-representation $\pi$ of $C^*(G)$ restricts to a $^*$-representation of $\C[G]$, and if this restriction is faithful, then it creates a $C^*$-norm  $\| \, \cdot \, \|_\pi$ on this algebra. Each $C^*$-norm on $\C[G]$ arises in this way, where by a $C^*$-norm on $\C[G]$ we mean a (faithful) norm such that the completion of $\C[G]$ with respect to this norm is a \Cs. 

Given a  unitary representation $\pi$ of $G$, we let $C^*_\pi(G)$ denote the completion of $\pi(\C[G])$. This is equal to the completion of $\C[G]$ with respect to the norm $\| \, \cdot \, \|_\pi$, if $\pi$ is faithful on $\C[G]$. The \emph{reduced group \Cs{}} of $G$, denoted by $C^*_\lambda(G)$, arises in this way from the left-regular representation $\lambda$ of $G$ on $\ell^2(G)$. It is well-known that the maximal and the reduced $C^*$-norms on $\C[G]$ are equal, i.e., $C^*(G) = C^*_\lambda(G)$, if and only if $G$ is amenable (see \cite[Theorem 2.6.8]{BO-approx}). It is also well-known (see, e.g., \cite[Exercise 6.3.3]{BO-approx}) that if the reduced group \Cs{} $C^*_\lambda(G)$ has a finite-dimensional representation, then $G$ must be amenable. Hence the following holds:

\begin{proposition} Let $G$ be a group and suppose that $C^*_\lambda(G)$ is just-infinite. Then either $C^*_\lambda(G)$ is simple, or $G$ is amenable.
\end{proposition}

\noindent
Whereas $\C[G]$ always has one maximal $C^*$-norm, there may or may not be a minimal $C^*$-norm on $\C[G]$, depending on the group $G$. If $G$ is $C^*$-simple, i.e., if  $C^*_\lambda(G)$ is a simple \Cs, then the norm $\| \, \cdot \, \|_\lambda$ on $\C[G]$  is minimal.

\begin{proposition} \label{prop:min-norm}
Let $G$ be a group, and let $\pi$ be a representation of $G$ which gives a faithful representation of $\C[G]$. If $C_\pi^*(G)$ is just-infinite, then $\| \, \cdot \, \|_\pi$ is a minimal $C^*$-norm on $\C[G]$.
\end{proposition}

\begin{proof} Any $C^*$-norm on $\C[G]$ which is smaller than $\| \, \cdot \, \|_\pi$ arises from a unitary representation $\nu$ of $G$ on a Hilbert space, which factors through $C^*_\pi(G)$. Since $\nu$ is injective on $\C[G]$, the image $\nu(C^*_\pi(G))$ cannot be finite dimensional, so $\nu$ is injective, and hence isometric, on $C^*_\pi(G)$. (Recall that each injective \sh{} between \Cs s automatically is isometric.) The norm arising from $\nu$ is therefore equal to the norm arising from $\pi$. 
\end{proof}

\noindent If $G$ is infinite and if $C^*_\pi(G)$ is simple, for some unitary representation $\pi$ of the group $G$, then $C^*_\pi(G)$ is just-infinite and $\| \, \cdot \, \|_\pi$ is a minimal norm on $\C[G]$; cf.\ Proposition~\ref{prop:min-norm}.

The group algebra $\C[G]$ is said to be \emph{$^*$-just-infinite}  if each $^*$-representation of $\C[G]$ either is injective, or has finite dimensional image. Note that $^*$-just-infinite is a formally weaker condition than ``just-infinite'', as $\C[G]$ can have non-self-adjoint two-sided ideals.

\begin{proposition} \label{prop:C*Gji}
Let $G$ be an infinite group. Then $C^*(G)$ is just-infinite if and only if $\C[G]$ is $^*$-just-infinite and $\C[G]$ has a unique $C^*$-norm.
\end{proposition}

\begin{proof} Suppose first that $C^*(G)$ is just-infinite. Let $\pi$ be a unital $^*$-representation of $\C[G]$, and extend it to a $^*$-representation of $C^*(G)$. Then $\pi$ is either injective on $C^*(G)$, or $\pi(C^*(G))$ is finite dimensional. If $\pi$ is  injective on $C^*(G)$, then it is also injective on $\C[G]$, while if $\pi(C^*(G))$ is finite dimensional, then so is $\pi(\C[G])$. Hence $\C[G]$ is $^*$-just-infinite.
Each $C^*$-norm on $\C[G]$ arises as $\| \, \cdot \, \|_\pi$, for some $^*$-representation $\pi$ of $C^*(G)$ which is faithful on $\C[G]$. Thus $\pi(C^*(G))$ is infinite dimensional, so $\pi$ must be injective on $C^*(G)$. This entails that $\| \, \cdot \, \|_\pi$ is the  maximal norm on $\C[G]$, and thus the only $C^*$-norm on $\C[G]$. 

Suppose now that $\C[G]$ has a unique $C^*$-norm, and that $\C[G]$ is $^*$-just-infinite. Let $\pi$ be a non-faithful unital $^*$-representation of $C^*(G)$. If the restriction of $\pi$ to $\C[G]$ were faithful, then it would induce a $C^*$-norm on $\C[G]$, which by uniqueness would be equal to the maximal $C^*$-norm on $\C[G]$. This contradicts that $\pi$ is non-faithful on $C^*(G)$. Hence $\pi$ is not faithful on $\C[G]$, whence $\pi(\C[G])$ is finite dimensional. In this case, $\pi(C^*(G))$ is equal to $\pi(\C[G])$. This  proves that $C^*(G)$ is just-infinite.
\end{proof}

\begin{corollary} Let $G$ be a group for which $C^*(G)$ is just-infinite. Then $G$ is amenable, and hence $C^*(G) = C^*_\lambda(G)$ is nuclear.
\end{corollary}

\begin{proof} It follows from Proposition~\ref{prop:C*Gji} that the reduced and the maximal norm on $\C[G]$ coincide, so $G$ is amenable. 
\end{proof}

\begin{corollary} \label{cor:3conditions} For each group $G$, if $C^*(G)$ is just-infinite, then  $\C[G]$ is $^*$-just-infinite, which in turn implies that $G$ is just-infinite.
\end{corollary}

\begin{proof}
The first implication follows from Proposition~\ref{prop:C*Gji}. To see that the second implication holds, suppose that $\C[G]$ is $^*$-just-infinite, and let $N$ be a non-trivial normal subgroup of $G$. The quotient map $G \to G/N$ lifts to a necessarily non-injective $^*$-homomorphism $\C[G] \to \C[G/N]$. Hence $\C[G/N]$ must be finite dimensional, whence $G/N$ is finite. 
\end{proof}

\noindent None of the reverse implications above hold; cf.\ Examples~\ref{ex:Z}  and \ref{ex:G}.

\begin{example} \label{ex:Z} The group algebra $\C[\Z]$ is $^*$-just-infinite, and the group  $\Z$ is just-infinite; but  $C^*(\Z)$ is not just-infinite, and $\C[G]$ has no minimal $C^*$-norm.
\end{example}

\begin{proof} Each unital  $^*$-representation $\pi$ of $\C[\Z]$ on a Hilbert space $H$ admits a natural factorization $\C[\Z] \to C(K) \to B(H)$, where $K  \subseteq \T$ is the spectrum of the unitary operator $u = \pi(1)$, and where $C(K) \to B(H)$ is injective.  It is easy to see that $\pi$ is faithful on $\C[\Z]$ if and only if $K$ is an infinite set. If $\pi$ is not faithful, then $K$ is finite, which entails that $\pi(\C[\Z])$ is finite dimensional. This shows that $\C[\Z]$ is $^*$-just-infinite.

As there is no minimal closed infinite subset of $\T$, there is no minimal $C^*$-norm on $\C[\Z]$, and we conclude from Proposition~\ref{prop:C*Gji}  that $C^*(\Z)$ is not just-infinite. This conclusion also follows from Example~\ref{ex:ji-abelian}.
\end{proof}

\begin{proposition} \label{prop:locfinite}
If $G$ is a locally finite group, then $\C[G]$ has a unique $C^*$-norm. 
\end{proposition}

\begin{proof} Each element $x \in \C[G]$ is a linear combination of finitely many elements from $G$, and each finitely generated subgroup of $G$ is finite, by assumption. Hence there is a finite subgroup $H$ of $G$ such that $x \in \C[H] \subseteq \C[G]$. Now, $\C[H]$ is a (finite dimensional) \Cs, so it has a unique $C^*$-norm. Thus any two $C^*$-norms on $\C[G]$ must agree on $x$. As $x$ was arbitrarily chosen, we conclude that $\C[G]$ has a unique $C^*$-norm.
\end{proof}

\begin{question} \label{q:locfin}
Let $G$ be a group and suppose that $\C[G]$ has a unique $C^*$-norm. Does it follow that $G$ is locally finite?
\end{question}

\noindent The \emph{augmentation ideal} of the full group \Cs{} $C^*(G)$ of a group $G$ is the kernel of the trivial representation $C^*(G) \to \C$. If $G$ is infinite  and if the augmentation ideal is simple, or, more generally, just-infinite, then $C^*(G)$ is just-infinite by Proposition~\ref{permanence} (ii), since the augmentation ideal always is essential when $G$ is infinite. 

There are locally finite groups whose augmentation ideal is simple, such as Hall's universal groups, see \cite{BHPS-1976} and \cite{LeiPug-IJM2002}. It follows from Theorem~\ref{thm:types}  that $C^*(G)$ is just-infinite of type ($\beta$), for any such group $G$. It is easy to see that if an amenable group $G$ has simple augmentation ideal, then it must be simple; however, simple groups (even locally finite ones) need not have simple augmentation ideal: the infinite alternating group $A_\infty$ is a counterexample.

\begin{lemma} \label{RF->RFD} Let $G$ be a residually finite group for which $C^*(G)$ is just-infinite. Then $C^*(G)$ is RFD (and hence of type ($\gamma$); cf.\ Theorem~\ref{thm:types}).
\end{lemma}

\begin{proof} Let $\{N_i\}_{i \in \I}$ be a decreasing net of finite index normal subgroups of $G$ with $\bigcap_{i \in \I} N_i = \{e\}$, and consider the \sh{} 
$$\Phi \colon C^*(G) \to \prod_{i \in \I} C^*(G/N_i).$$ 
It suffices to show that $\Phi$ is injective; and by the assumption that $C^*(G)$ is just-infinite, it further suffices to show that the image of $\Phi$ is infinite dimensional. The latter follows from the fact that $G$ is infinite (as $C^*(G)$ is just-infinite) and (hence) that  $\sup_{i \in \I} |G:N_i| = \infty$.
\end{proof}

\begin{question} \label{q:ji}
Does there exist an infinite, residually finite group $G$ such that $C^*(G)$ is just-infinite?
\end{question}

\noindent If such a group $G$ exists, then $C^*(G)$ will be a RFD just-infinite  \Cs{} by Lemma~\ref{RF->RFD}. If the answer to Question~\ref{q:locfin} is affirmative, then $G$ must be locally finite. This leads to the following:

\begin{question} \label{q:jilf}
Does there exist an infinite, residually finite, locally finite (necessarily just-infinite) group $G$ such that $\C[G]$ is $^*$-just-infinite?
\end{question}

\noindent If such a group $G$ exists, then $C^*(G)$ will be a RFD just-infinite  \Cs{} by Lemma~\ref{RF->RFD}, Proposition~\ref{prop:locfinite} and Proposition~\ref{prop:C*Gji}. 
\emph{After the first version of this paper was made public, Question~\ref{q:jilf} has been answered in the affirmative in \cite{BGS-2016}.}

Just-infinite groups are divided into three disjoint subclasses (the trichotomy for just-infinite groups), see \cite[Section 6]{Gr00}: The \emph{non-residually finite} ones (which contain a finite index normal subgroup $N$ which is the product of finitely many copies of a simple group), \emph{branch groups} (see more about those in Theorem~\ref{thm:no-branch} below), and the \emph{hereditarily just-infinite groups}, i.e., the residually finite groups for which all finite index normal subgroups are just-infinite. It is shown in Theorem~\ref{thm:no-branch} below that if $G$ is a just-infinite branch group, then $\C[G]$ is not $^*$-just-infinite, whence $C^*(G)$ is not just-infinite. Hence, if there exists a residually finite group $G$ for which $C^*(G)$ is just-infinite (and hence also RFD), then $G$ must be hereditarily just-infinite. 

Consider the following three (classes of) examples of hereditarily just-infinite groups: the integers $\Z$, the infinite dihedral group $D_\infty$, and $\mathrm{PSL}_n(\Z)$, for $n \ge 3$. As shown below, if $G$ is any of these groups, then $C^*(G)$ is not just-infinite. Moreover,  there is no unitary representation $\pi$ of $G$ such that $C^*_\pi(G)$ is RFD and just-infinite. If $G = \Z$, then this claim follows immediately from Example~\ref{ex:ji-abelian}.  In the two examples below we discuss the situation for the two other (classes of) groups.

\begin{example}[$\mathrm{PSL}_n(\Z)$, $n \ge 3$] \label{ex:PSLnZ}
The groups $\mathrm{PSL}_n(\Z)$, $n \ge 3$,  are renowned for being the first examples of  infinite groups with Kazdan's property (T), as first shown by Kazdan. For a different and nice proof by Shalom, see \cite{Shalom-1999}. They are residually finite, as witnessed by the finite quotient groups $\mathrm{PSL}_n(\Z/N\Z)$, $N \in \N$; and they are   hereditarily just-infinite by Margulis' normal subgroup theorem. Bekka--Cowling--de la Harpe proved in \cite{BekCowHar:PSLnZ} that $\mathrm{PSL}_n(\Z)$ is $C^*$-simple, for all $n \ge 2$. In particular, $\mathrm{PSL}_n(\Z)$  is an ICC group (all its non-trivial conjugacy classes are infinite). 

We conclude from these facts that the \Cs{} $C^*_\lambda(\mathrm{PSL}_n(\Z))$ is just-infinite (being simple) for all $n \ge 2$, while the full group \Cs{} $C^*(\mathrm{PSL}_n(\Z))$ is not just-infinite, because $\mathrm{PSL}_n(\Z)$ is non-amenable, for $n \ge 2$. 

Bekka proved in \cite{Bekka-Invent-07} that the set of extremal characters on $\mathrm{PSL}_n(\Z)$, for $n \ge 3$, is a countably infinite set consisting of the trivial character $\delta_e$ and a sequence $\{\delta_k\}_{k=1}^\infty$ of characters, each of which factors through a finite quotient,  $\mathrm{PSL}_n(\Z/N\Z)$, of $\mathrm{PSL}_n(\Z)$ for a suitable integer $N$ (depending on $k$). Recall that each (extremal) character on a group corresponds to an extremal trace on its full group \Cs. The trivial character $\delta_e$ on $\mathrm{PSL}_n(\Z)$ corresponds to the canonical trace $\tau_0$ on $C^*(\mathrm{PSL}_n(\Z))$; while for $k \ge 1$, the character $\delta_k$ corresponds to a trace, denoted by $\tau_k$, whose GNS-representation $\pi_{\tau_k}$ is finite dimensional. Bekka also shows that  $\tau_k \to \tau_0$ in the weak$^*$ topology.

Furthermore, observe that $C^*(\mathrm{PSL}_n(\Z))$ has a just-infinite quotient, namely the simple \Cs{} $C^*_\lambda(\mathrm{PSL}_n(\Z))$. However, as shown below, there is no unitary representation $\pi$ of $\mathrm{PSL}_n(\Z)$ such that $C^*_\pi(\mathrm{PSL}_n(\Z))$ is RFD and just-infinite. 

Indeed, assume that  $\pi$ is such a unitary representation of $\mathrm{PSL}_n(\Z)$.   As in Remark~\ref{rem:add-structure}, let $\{\pi_j\}_{j=1}^\infty$ be an exhausting sequence of pairwise  inequivalent non-faithful irreducible representations of $C^*_\pi(\mathrm{PSL}_n(\Z))$. Then $\rho_j = \pi_j \circ \pi$, $ j \ge 1$, is a sequence of pairwise  inequivalent non-faithful irreducible representations of $C^*(\mathrm{PSL}_n(\Z))$. Hence $\rho_j$ is equivalent to $\pi_{\tau_{k(j)}}$, for some $k(j) \ge 1$, by the above mentioned result of Bekka.  Suppose now that $x \in C^*(\mathrm{PSL}_n(\Z))$ belongs to the kernel of $\pi$. Then, for all $j \ge 1$, $\pi_{\tau_{k(j)}}(x) = 0$, so $\tau_{k(j)}(x^*x) = 0$. It follows that $\tau_0(x^*x) = \lim_{j\to\infty} \tau_{k(j)}(x^*x) = 0$, so $\lambda(x) = 0$. This shows that $\lambda$ is weakly contained in $\pi$. We conclude that the left-regular representation $\lambda$ factors through $\pi$, so the simple \Cs{} $C^*_\lambda(\mathrm{PSL}_n(\Z))$ is a quotient of $C^*_\pi(\mathrm{PSL}_n(\Z))$. Each simple quotient of a RFD just-infinite  \Cs{} is finite dimensional, so $C^*_\pi(\mathrm{PSL}_n(\Z))$ cannot both be RFD and just-infinite. 
\end{example}

\begin{example}[The infinite dihedral group $D_\infty$] \label{ex:Dinfty}
The infinite dihedral group $D_\infty$ is an example of a hereditarily residually finite just-infinite group,  see \cite{Gr00}, and it is isomorphic to the free product $\Z_2 * \Z_2$, which is an amenable group (of linear growth). The group \Cs{} $C^*(\Z_2 * \Z_2)$ is known to be a sub-\Cs{} of $M_2(C([0,1]))$ (being the universal unital \Cs{} generated by two projections), and is hence subhomogeneous (cf.\ Proposition~\ref{prop:subhom}). Clearly, any quotient of a subhomogeneous \Cs{} is again subhomogeneous, so we conclude from Lemma~\ref{lm:subhom-justinf} that $C^*(\Z_2 * \Z_2)$ is not just-infinite, and neither is any of its quotients. 
\end{example}

\section{Algebras associated with  groups of intermediate growth} \label{sec:G}

\noindent In this section we present some results concerning algebras associated with the $3$-generated infinite torsion group constructed in \cite{Gr80}, which we here will denote by $\cG$. This group is a simple example of a group of Burnside-type, and it is investigated more deeply in \cite{Gr84} and many other papers (see the surveys \cite{Gr05}, \cite{Gr11}, and the references therein). Among its unusual properties, most notably $\cG$ is of intermediate growth (between polynomial and exponential), and, as a consequence, it is amenable, but not elementary amenable, thus answering questions by Milnor and Day, respectively;  cf.\ \cite{Gr84}. Furthermore, $\cG$ is a just-infinite group of branch type (and hence residually finite), and moreover, it is a self-similar group (i.e., a group generated by states of a Mealy-type automaton).

There are indications that various \Cs s associated with $\cG$ (including $C^*(\cG)$ and some of its quotients, discussed below) may be new types of \Cs s with properties unseen yet in the theory of operator algebras. Our main conjecture in this direction is the following:

\begin{conjecture} \label{conjecture}
The (self-similar RFD) \Cs{} $C^*_\pi(\cG)$ generated by the Koopman representation $\pi$ of $\cG$ is just-infinite. 
\end{conjecture}

\noindent
The  Koopman representation $\pi$ of $\cG$ will be described below, along with the notion of self-similarity. If the conjecture above is correct, then $C^*_\pi(\cG)$ is a just-infinite \Cs{} of type ($\gamma$) as described in Theorem~\ref{thm:types}; cf.\ Lemma~\ref{RF->RFD}, and it is the first example of such a \Cs{} associated with a group. 

Recall that the group $\cG$ is generated by four elements $a,b,c,d$ satisfying the relations 
\begin{equation} \label{eq:G-relators}
 1 = a^2=b^2=c^2=d^2 = bcd = \sigma^k\big((ad)^4\big) = \sigma^k\big((adacac)^4\big), 
\end{equation}
for $k=0,1,2, \dots$, where the permutation $\sigma$ is given by the substitutions:
$$a \to aca, \qquad b \to d, \qquad c \to b, \qquad d \to c.$$
This presentation was found by I.\ Lysenok in \cite{Ly85}, and it is a minimal presentation (in the sense that no relator in \eqref{eq:G-relators} can be deleted without changing the group, see \cite{Gr99}). In fact, $\cG$ is generated by $3$ elements, as $d = bc$. The set $\{1,b,c,d\}$ is the Klein group $\Z/2\Z \oplus \Z/2\Z$. 

For our purposes it will be most important to know that $\cG$ has a faithful self-similar action by automorphisms on an infinite binary rooted tree $T$, as shown, in part, here:

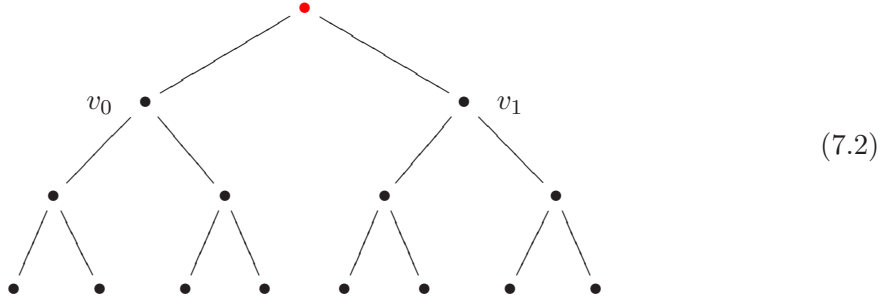
\begin{equation} \label{eq:T}
\begin{split}
\xymatrix@C.3pc{&&&&&&&\rbullet \ar@{-}[dllll] \ar@{-}[drrrr]&&&&&&&\\
&&v_0& \bullet \ar@{-}[dll] \ar@{-}[drr] &&&&&&&& \bullet \ar@{-}[dll] \ar@{-}[drr] &v_1&&\\
&\bullet \ar@{-}[dl] \ar@{-}[dr] &&&&\bullet \ar@{-}[dl] \ar@{-}[dr]  &&&& \bullet \ar@{-}[dl] \ar@{-}[dr]  &&&& \bullet \ar@{-}[dl] \ar@{-}[dr]  & \\ 
\bullet && \bullet && \bullet && \bullet && \bullet && \bullet && \bullet && \bullet} 
\end{split}
\end{equation}
The generators $a,b,c,d$ act on $T$ as follows: The root of the tree (marked in red) is a common fixed point. The generator $a$ just permutes the two vertices $v_0$ and $v_1$ at the first level and acts trivially inside the subtrees $T_0$ and $T_1$ with roots $v_0$ and $v_1$, respectively. The generators $b,c,d$ fix the vertices $v_0$ and $v_1$ (and hence leave the subtrees $T_0$ and $T_1$ invariant), and they are defined recursively by:
\begin{equation} \label{eq:2}
b|_{T_0} = a, \quad  b|_{T_1} = c, \qquad
c|_{T_0} = a, \quad  c|_{T_1} = d, \qquad
d|_{T_0} = 1, \quad  d|_{T_1} = b,
\end{equation}
when identifying the subtrees $T_0$ and $T_1$ with $T$ in the natural way, and where $1$ stands for the identity automorphism. For more details on this definition, and other definitions of $\cG$, we refer to \cite{Gr80, Gr84, Gr05, Gr11}. The relations \eqref{eq:2} imply that $\cG$ is a self-similar group in the sense that it has a natural embedding
\begin{equation} \label{eq:psi-G}
\psi\colon \cG \to \cG \wr (\Z/2\Z) \cong (\cG \times \cG) \rtimes (\Z/2\Z),
\end{equation}
where $\Z/2\Z = \{e,\varepsilon\}$ acts on $\cG \times \cG$ by permuting the two copies of $\cG$ ($e$ is the identity element and $\varepsilon$ is a transposition). The embedding $\psi$ is given as follows: 
$$\psi(a) = (1,1) \cdot \varepsilon = \varepsilon, \hspace{.2cm} \psi(b) = (a,c) \cdot e = (a,c), \hspace{.2cm} \psi(c) = (a,d) \cdot e = (a,d), \hspace{.2cm} \psi(d) = (1,b) \cdot e =(1,b).$$

To further illlustrate this action of $\cG$ on the tree $T$ it is convenient to label the vertices of the $n$th level, $V_n$, of $T$ by the set $\{0,1\}^n$ and equip each $V_n$ with the lexicographic ordering:
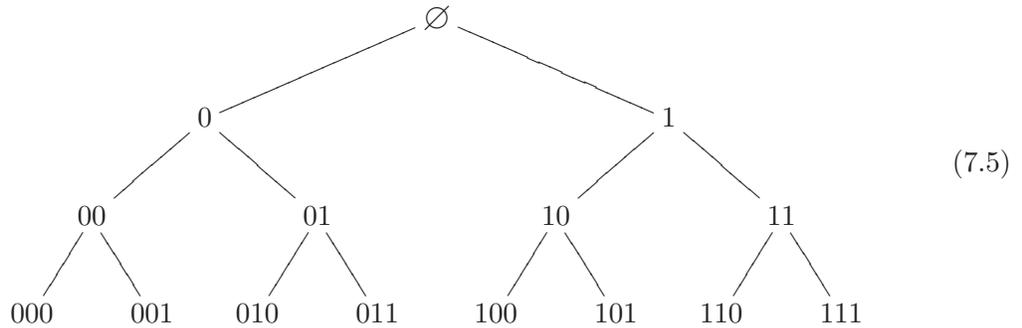
\begin{equation} \label{eq:T1}
\begin{split}
\xymatrix@C.25pc{&&&&&&&\emptyset \ar@{-}[dllll] \ar@{-}[drrrr]&&&&&&&\\
&&& 0 \ar@{-}[dll] \ar@{-}[drr] &&&&&&&& 1\ar@{-}[dll] \ar@{-}[drr] &&&\\
&00 \ar@{-}[dl] \ar@{-}[dr] &&&&01 \ar@{-}[dl] \ar@{-}[dr]  &&&& 10 \ar@{-}[dl] \ar@{-}[dr]  &&&& 11 \ar@{-}[dl] \ar@{-}[dr]  & \\
000 && 001 && 010 && 011 && 100 && 101 && 110 && 111} 
\end{split}
\end{equation}
The action of the group $\cG$ on $T$ yields an action of $\cG$ by homeomorphisms on the boundary $\partial T$ of $T$, which consists of geodesic rays joining the root $\emptyset$ with infinity. The boundary $\partial T$ can in a natural way be identified with the Cantor set $\{0,1\}^\N$ of infinite binary sequences equipped with the Tychonoff topology. 

Let $\mu = \bigtimes_{n=1}^\infty \, \{\frac12,\frac12\}$ be the uniform Bernoulli measure on $\partial T$. It is invariant with respect to the action of the entire group $\mathrm{Aut}(T)$ of automorphisms on $T$, and hence with respect to the action of $\cG$ on $T$. The topological dynamical system $(\cG, \partial T)$ can be converted into a metric dynamical system $(\cG,\partial T, \mu)$ which is ergodic (while $(G,\partial T)$ is minimal), because the action of $\cG$ on each level $V_n$ is transitive, see \cite[Proposition 6.5]{GNS00}. 

Let $\pi$ be the (unitary) Koopman representation of $\cG$ on the Hilbert space $L^2(\partial T,\mu)$ given by $\big(\pi(g)f\big)(x) = f(g^{-1}x)$, where $f \in L^2(\partial T,\mu)$, $g \in \cG$, and $x \in \partial T$.
We denote the image of the group algebra $\C[\cG]$ under the representation $\pi$ by $B$, and we let as usual $C^*_\pi(\cG)$ denote the completion of $\C[\cG]$ with respect to the norm induced by $\pi$.

The following theorem carries some evidence in support of Conjecture~\ref{conjecture}.

\begin{theorem} \label{thm:B} Let $\cG = \langle a,b,c,d \rangle$ be the infinite torsion group of intermediate growth from above, let $\pi$ be the Koopman representation of $\cG$,  and let $B=\pi(\C[\cG])$. Then:
\begin{enumerate}
\item $B$ is self-similar, infinite dimensional and RFD.
\item $C^*_\pi(\cG)$ is self-similar, infinite dimensional, RFD, and it posseses a faithful trace.
\item The natural surjection $\pi \colon \C[\cG] \to B$ is not injective, whence $\C[\cG]$ is not $^*$-just-infinite.
\item $B$ is just-infinite.
\end{enumerate}
\end{theorem}

\noindent The notions of self-similarity of the algebras $B$ and $C^*_\pi(\cG)$ will be explained below. Theorem~\ref{thm:B} above is proved at the end of this section. 

The  type of just-infinite algebras (also called ``thin algebras'') considered above were studied for the first time by  Sidki in  \cite{Sidki:Burnside}.   The  group  used  by  Sidki  was the Gupta--Sidki  3-group $H$, and  the  algebra  was  defined  over  a  field ${\mathbb{F}}_3$   in a rather involved way as   a certain  inductive  limit.  If  one  considers the ``Koopman"  representation of  $H$ over the field $\mathbb{F}_3$,  then  the  image  of the group  algebra ${\mathbb{F}}_3[H]$ will  be   isomorphic  to  Sidki's thin  algebra.

The \Cs{}  generated  by  the Koopman  representation of the group $\cG$ (considered in this section) was considered in \cite{BG00}, and so was the algebra  $B$, even though it was not explicitly defined there.
Vieira, \cite{Vieira-2001},  used  Sidki's  approach  to  define a ``thin  algebra''  of  the group $\cG$ over the field $\mathbb{F}_2$,  and proved  that  it  is  just-infinite.

Thin  algebras  under  the  name  ``Tree  enveloping  algebras"   were  considered  by  Bartholdi  in \cite{Bartholdi-IJM-2006}.  He  defines  algebras,  similar to the algebra $B$ in Theorem~\ref{thm:B}, however, over  arbitrary   fields.  He  considers  a  vector  space   with a  basis  consisting  of  all  points   of  the  boundary  of  the  rooted  tree,  and  then  defines an algebra  as the  image  of  the group  algebra  in  the  algebra  of   endomorphisms  of  this  huge  vector  space.
One  can  show   that  if  the  field  is   complex  numbers  and  the  group  is  the group $\cG$,  then   Bartholdi's algebra  is  isomorphic  to  the algebra $B$ we are considering here.

In \cite[Theorem  3.9]{Bartholdi-IJM-2006},   a  sufficient  condition  is given for the tree  enveloping   algebra  to  be   just-infinite.   This  condition  is  satisfied  in  the  case  of  the group $\cG$.

\begin{example} \label{ex:G}
As mentioned above, the group $\cG$ is just-infinite. We can therefore deduce from Theorem~\ref{thm:B}(iii) that just-infiniteness of a group $\cG$ does not imply that its complex group algebra $\C[\cG]$ is $^*$-just-infinite.
\end{example}

\subsection*{Self-similarity of graphs, Hilbert spaces, representations and algebras}

\noindent Let $X = \{x_1,x_2, \dots, x_d\}$ be an alphabet on $d \ge 2$ letters, let $X^* = \bigsqcup_{n=0}^\infty X^n$ be the set of words over $X$, and let $T = T_X$ be the $d$-arnery rooted tree whose vertices are in bijection with the elements of $X^*$ (so that the $n$th level $V_n$ of $T$ corresponds to $X^n$). The action of an arbitrary group $G$ on $T$ by automorphisms induces an action $G \curvearrowright X^*$. This action is said to be \emph{self-similar} if for all $g \in G$ and all $x \in X$, there are $h \in G$ and $y \in X$ such that $g(xw) = yh(w)$, for all $w \in X^*$. If this holds, then for every $v \in X^*$, there exists a unique $h \in G$ satisfying, for all $w \in X^*$,
\begin{equation} \label{eq:self-similar}
g(vw) = g(v)h(w).
\end{equation}
The element $h$ is called the section (or restriction) of $g$ in $v$, and is denoted by $h = g|_v$. For example, for the group $\cG = \langle a,b,c,d \rangle$ under examination, we have; cf.\ \eqref{eq:2}, that
$$a|_{v_0} = a|_{v_1} = 1, \qquad b|_{v_0} = a, \quad b|_{v_1} = c, \qquad c|_{v_0} = a, \quad c|_{v_1} = d, \qquad d|_{v_0} = 1, \quad d|_{v_1} = b.$$

Let $H$ be a separable infinite dimensional Hilbert space, and fix an integer $d \ge 2$. A unitary operator $u \colon H \to H^d = H \oplus H \oplus \cdots \oplus H$ is called a \emph{$d$-similarity} of $H$. Each $d$-similarity arises from $d$ isometries $s_1, \dots, s_d$ on $H$, satisfying $\sum_{j=1}^d s_js_j^* = 1$, as follows:
$$u(\xi) = (s_1^*\xi, \dots, s_d^*\xi), \qquad u^*(\xi_1, \dots, \xi_d) = \sum_{j=1}^d s_j \xi_j, 
$$
for $\xi, \xi_1, \dots, \xi_d \in H$. Observe that $s_1, \dots, s_d$ define a representation of the Cuntz algebra $\cO_d$, and that every representation of $\cO_d$ is obtained in this way.  For each $x = (x_{i_1},x_{i_2}, \dots, x_{i_n}) \in X^*$ consider the isometry on $H$ given by $S_x = s_{i_1}s_{i_2} \dots s_{i_n}$, and observe that $S_xS_y = S_{xy}$.  A unitary representation $\rho$ of a group $G$ on a Hilbert space $H$ is said to be \emph{self-similar} with respect to the $d$-similarity $\psi$ considered in \eqref{eq:psi-G} above, if
\begin{equation} \label{eq:rho-ss}
\rho(g)S_x = S_{y} \rho(h),
\end{equation}
for all $g,h \in G$ and all $x,y \in X^*$ satisfying $g(xw) = yh(w)$, for all $w \in X^*$. In other words,
$\rho(g)S_x = S_{g(x)}\rho(g|_x)$, for all $g \in G$ and $x \in X^*$. 

The image $B_\rho = \rho(\C[G])$, where $\rho$ is a self-similar representation, is called a \emph{self-similar} (abstract) \emph{algebra}. The \Cs{} $C^*_\rho(G)$ associated with a self-similar representation $\rho$ is called a \emph{self-similar \Cs}. One of the features of  the self-similar algebra $B_\rho$ (or of the \Cs{} $C^*_\rho(G)$) is the existence of the unital embedding
\begin{equation} \label{eq:psi}
\psi_\rho \colon B_\rho \to M_d(B_\rho), \qquad b \mapsto \begin{pmatrix} s_1^*bs_1 & \cdots & s_1^*bs_d \\ \vdots & &\vdots\\ s_d^*bs_1 & \cdots&  s_d^*bs_d  \end{pmatrix}, 
\end{equation}
for $b \in B_\rho$. It follows from \eqref{eq:rho-ss} that $s_j^*B_\rho s_i \subseteq B_\rho$, for all $i,j$. The embedding $\psi_\rho$ is typically not surjective. Nonetheless, it has many interesting and non-trivial features, see, for example, Lemma~\ref{Lemma2} below. 

In the case of our main example $\cG = \langle a,b,c,d \rangle$ and of the Koopman representation $\pi$ of $\cG$ on $H=L^2(\partial T, \mu)$, we have an explicit self-similarity $H \to H \oplus H$ arising from the two isometries $s_0,s_1$ on $H$ defined by
\begin{equation} \label{eq:s_i}
(s_if)(x) = f(ix), \qquad
\end{equation}
for $i=1,2$, where  $f \in L^2(\partial T, \mu)$ and $x \in \partial T$, and where $ix \in \partial T = \{0,1\}^\N$ is the word obtained by putting the letter $i$ in front of the word $x$. The resulting embedding $\psi_\pi \colon B \to M_2(B)$ is given as follows on the generators:
\begin{equation} \label{eq:8}
\psi_\pi(\bar{a}) = \begin{pmatrix} 0 & 1 \\ 1 & 0 \end{pmatrix}, \quad \psi_\pi(\bar{b}) = \begin{pmatrix} \bar{a} & 0 \\ 0 & \bar{c} \end{pmatrix}, \quad \psi_\pi(\bar{c}) = \begin{pmatrix} \bar{a} & 0 \\ 0 & \bar{d} \end{pmatrix}, \quad \psi_\pi(\bar{d}) = \begin{pmatrix} 1 & 0 \\ 0 & \bar{b} \end{pmatrix},
\end{equation}
(as can be deduced from \eqref{eq:rho-ss} and \eqref{eq:psi}), where we have introduced the notation $\bar{g} = \pi(g)$, for $g \in \cG$. (The Koopman representation is faithful on $\cG$, so the map $g \mapsto \bar{g}$ is injective, but the Koopman representation is not faithful on $\C[\cG]$; cf.\ Theorem~\ref{thm:B}, so it is pertinent to distinguish between $g$ and $\pi(g)$.)

\subsection*{More on the Koopman representation}

\noindent What we are going to present here is known in the more general situation of groups acting on rooted trees, \cite{BG00, BeHa03, Gr05}. Consider the binary rooted tree $T$ (as described in \eqref{eq:T} and \eqref{eq:2}), and the Koopman representation $\pi$  of the group $\cG = \langle a,b,c,d \rangle$ on $L^2(\partial T, \mu)$. 

For each $n \ge 1$, let $v_{n,1}, v_{n,2}, \dots, v_{n,2^n}$ be the order preserving enumeration of the set $V_n = \{0,1\}^n$ (equipped with the lexicographic ordering); cf.\ \eqref{eq:T1}, and write $\partial T = \bigsqcup_{i=1}^{2^n} E_{n,i}$, where $E_{n,i}$ is the set of infinite words in $\partial T = \{0,1\}^\N$ that start with $v_{n,i}$. Set
$$H_n = \Span \{\chi_{E_{n,i}} \mid i = 1,2, \dots, 2^n\} \subseteq H = L^2(\partial T,\mu),$$
which is a subspace of dimension $2^n$. Since $E_{n,i} = E_{n+1,2i-1} \cup E_{n+1,2i}$, we see that $H_n \subseteq H_{n+1}$. Moreover, as the cylinder sets $E_{n,i}$, $n \ge 1$, $1 \le i \le 2^n$, form a basis for the topology on $\partial T$, it follows that $\bigcup_{n=1}^\infty H_n$ is dense in $H$. 

The subspaces $H_n$ are $\pi$-invariant. Let $\pi_n$ be the restriction of $\pi$ to $H_n$, for $n \ge 1$. Observe that $\pi_n$ is unitarily equivalent to the representation of $\cG$ on $\ell^2(V_n)$ arising from its action on the $n$th level $V_n$ of the tree $T$. More specifically, identify $H_n$ with $\ell^2(V_n)$ via the isomorphism that identifies $\chi_{E_{n,i}}$ with $\delta_{v_{n,i}}$. Write $H_{n+1} = H_n \oplus H_n^\perp$, and let $\pi_n^\perp$ denote the restriction of $\pi$ to $H_n^\perp$. Note that $H_n^\perp$ has dimension $2^n$. It is shown in the appendix of \cite{BeHa03} that the representation $\pi_n^\perp$ of $\cG$ is irreducible, for each $n \ge 1$. Thus we have decompositions
\begin{equation} \label{eq:K}
H = \C \oplus \bigoplus_{n=0}^\infty H_n^\perp, \qquad \pi = {\bf{1}} \oplus \bigoplus_{n=0}^\infty \pi_n^\perp,
\end{equation}
of the Hilbert space $H$ and of the representation $\pi$ into irreducible representations, where we identify $H_0$ with $\C$, and $\pi_0$ with the trivial representation $\bf{1}$. 

\subsection*{The proof of Theorem~\ref{thm:B}}

\noindent \emph{Proof of Theorem~\ref{thm:B} {\rm{(i)}}}:
Recall from \eqref{eq:s_i} that we have isometries $s_0,s_1$ on the Hilbert space $H = L^2(\partial T,\mu)$ satisfying the Cuntz relation $s_0s_0^* + s_1s_1^* = 1$. The range of the isometry $s_i$ is $L^2(\partial T_i, \mu_i)$, where $T_0$ and $T_1$ are the subtrees of $T$ with roots $v_0$ and $v_1$, respectively; cf.\ \eqref{eq:T}, and where $\mu_0$ and $\mu_1$ are the normalized restrictions of $\mu$ to the subsets $\partial T_0$ and $\partial T_1$, respectively, (making them probability measures). The Koopman representation $\pi$ is self-similar with respect to the $2$-similarity of $H$ given by the isometries $s_0,s_1$, so $B = \pi(\C[\cG])$ is self-similar.

By \eqref{eq:K} and irreducibility of the representations $\pi_n^\perp$, we see that $B$ is a subalgebra of
\begin{equation} \label{eq:MF}
M:= \C \oplus \prod_{n=0}^\infty M_{2^n}(\C),
\end{equation}
with the property that the projection of $B$ onto each summand in \eqref{eq:MF} is surjective. Hence $B$ is infinite dimensional and RFD. This completes the proof of (i).

\medspace
\noindent \emph{Proof of Theorem~\ref{thm:B} {\rm{(ii)}}}: It follows from \eqref{eq:K} that the inclusion of $B$ into $M$ is isometric, when $B$ is equipped with the norm arising from the Koopman representation $\pi$. Thus $C^*_\pi(\cG)$, which is the completion of $B$ with respect to this norm, embeds into $M$. Hence $C^*_\pi(\cG)$ is RFD. Moreover, it is infinite dimensional because it contains the infinite dimensional algebra $B$, and it is self-similar because the Koopman representation $\pi$ is self-similar. Finally, $M$ has a faithful trace, for example the one given by
$$\tau(x) = \sum_{j=-1}^\infty \alpha_j \tau_{j}(x_{j}),$$
where $x = (x_1,x_0,x_1,\dots) \in M$, $\tau_{n}$ is the normalized trace on $M_{2^n}(\C)$, for each $n \ge 1$ (and $\tau_{-1}$ and $\tau_0$ are the normalized traces on $\C$), and $\{\alpha_j\}_{j=-1}^\infty$ is any sequence of strictly positive numbers summing up to $1$. Hence $C^*_\pi(\cG)$ has a faithful trace, being a sub-\Cs{} of $M$.

\medskip
\noindent \emph{Proof of Theorem~\ref{thm:B} {\rm{(iii)}}}:
The first claim of (iii) is proved in the lemma below, and the second claim follows from the first claim and the fact, proved in (i), that $B$ is infinite dimensional. 

\medskip \noindent
The result below can be found in \cite{GrN07}. We include its  proof for completeness of the exposition.

\begin{lemma} \label{lm:z} 
$(1-d)a(1-d)$
is a non-zero element in the kernel of $\pi \colon \C[\cG] \to B$.
\end{lemma}

\begin{proof} We observe first that $z :=  a - da - ad + dad$ 
is non-zero in $\C[\cG]$. Indeed, if $z=0$, then $a+dad=da+ad$, which can happen only if either $a = da$ and $dad = ad$, or $a = ad$ and $dad = da$. Both are impossible, because $d \ne e$. (It is also easy to see, for example using the action of $\cG$ of the tree $T$, that the four elements $a, da,ad, dad$ are pairwise distinct.) 

By \eqref{eq:8} (and retaining the notation $\bar{g} = \pi(g)$, for $g \in \cG$), we have
$$\psi_\pi(\pi(z)) = \psi_\pi\big((1-\bar{d})\bar{a}(1-\bar{d})\big) = \begin{pmatrix} 0 & 0 \\ 0 & 1-\bar{b}\end{pmatrix}\begin{pmatrix} 0 & 1 \\ 1 & 0\end{pmatrix}\begin{pmatrix} 0 & 0 \\ 0 & 1-\bar{b}\end{pmatrix} =0,$$
where $\psi_\pi$ is the embedding of $B$ into $M_2(B)$ arising from self-similarity. 
As $\psi_\pi$ is injective, this implies that $\pi(z) = 0$.
\end{proof}

\noindent
The proof of part (iv) of Theorem~\ref{thm:B} is somewhat lengthy and is divided into  several lemmas. The proof mimics the proof of the fact that $\cG$ is a just-infinite group, as well as the idea from the proof of \cite[Theorem 4]{Gr00} showing that a proper quotient of an arbitrary branch group is virtually abelian. In our situation, the following statement from \cite[Proposition~8]{Gr00} is useful:

\begin{proposition} \label{prop:K} The normal subgroup $K = \langle (ab)^2 \rangle^\cG$ has finite index 16 in $\cG$, and it is of self-replicating type, written $K \times K \prec K$, i.e., $K \times K \subseteq \psi(K)$, where $\psi$ is given by \eqref{eq:psi-G}.
\end{proposition}

\noindent Let $\psi\colon \cG \to (\cG \times \cG) \rtimes \Z/2\Z$ be as defined in \eqref{eq:psi-G}. For each $m \ge 1$, the stabilizer subgroup $\mathrm{St}_\cG(m)$ of $\cG$, with respect to the action of $\cG$ on the tree $T$, is the set of elements $g \in \cG$ that fix all vertices at level $m$, i.e., all vertices in $V_m$. In particular, if $g \in \mathrm{St}_\cG(1)$, then $\psi(g) \in \cG \times \cG$. The group $K$ is a subgroup of $\mathrm{St}_\cG(1)$.

It is also shown in \cite{Gr80}  that $\cG$ itself is self-replicating (or recurrent), in the sense that $\mathrm{St}_\cG(1) \le_S  \cG \times \cG$, where $\le_S$ is \emph{subdirect product}. This means that the group homomorphisms
$$\xymatrix{\mathrm{St}_\cG(1) \ar[r]^\psi & \cG \times \cG \ar[r]^-{\pi_j} & \cG,} $$
where $\pi_j$, $j=0,1$, are the coordinate homomorphisms, are surjective.

Let $\Delta$ be the ideal in $B$ generated by the set $\{\bar{k}-1 \mid k \in K\} \subseteq B$, where $K$ is as in Proposition~\ref{prop:K} above. Then $B/\Delta$ has dimension at most 16. To see this, let $\{t_1, t_2,\dots, t_{16}\}$ be representatives of the cosets of $K$ in $\cG$. For each $g \in \cG$, there exist $i$ in $\{1,2, \dots, 16\}$ and $k$ in $K$ such that $g = t_i k = t_i + t_i(k-1)$, so $\bar{g} \in \bar{t_i} + \Delta$. This shows that $B/\Delta$ is the linear span of the elements $\bar{t_1} +\Delta, \bar{t_2} +\Delta, \dots, \bar{t}_{16}+\Delta$. 

Let $\psi_\pi \colon B \to M_2(B)$ be as defined in \eqref{eq:psi}, and let $\psi_\pi^n \colon B \to M_{2^n}(B)= B \otimes M_{2^n}(\C)$ denote the ``$n$th iterate of $\psi_\pi$'', in the sense that 
$$\psi_\pi^n = (\psi_\pi \otimes \mathrm{id}_{M_{2^{n-1}}(\C)}) \circ \cdots \circ  (\psi_\pi \otimes \mathrm{id}_{M_{2}(\C)})  \circ \psi_\pi.$$
 The homomorphisms $\psi_\pi^n$ are not surjective, but the following holds:

\begin{lemma} \label{Lemma2}
For each $n \ge 1$,  $M_{2^n}(\Delta) \subseteq \psi^n_\pi(\Delta)$.
\end{lemma}

\begin{proof}
The lemma follows easily by induction on $n$, once the base step $n=1$ has been verified. So let us show that $M_{2}(\Delta) \subseteq \psi_\pi(\Delta)$.

It follows from Proposition~\ref{prop:K}  that for each $k \in K$ we can find $k' \in K$ such that $\psi(k') = (k,1)$. Hence
\begin{equation} \label{eq:i}
\psi_\pi(\bar{k'}) = \begin{pmatrix} \bar{k} & 0 \\ 0 & 1 \end{pmatrix}, \qquad 
\psi_\pi(\bar{k'}-1) = \begin{pmatrix} \bar{k}-1 & 0 \\ 0 & 0 \end{pmatrix}.
\end{equation}
Let $x,x' \in \Delta$ be such that 
\begin{equation} \label{eq:ii}
\psi_\pi(x') = \begin{pmatrix} x & 0 \\ 0 & 0 \end{pmatrix}.
\end{equation}
Since $\cG$ is self-replicating; cf.\ the comments below Proposition~\ref{prop:K}, we can for each $f \in \cG$ find $g \in \mathrm{St}_\cG(1)$ and $h \in \cG$, such that $\psi(g) = (f,h)$. Then 
$$\psi_\pi(\bar{g}x') =  \begin{pmatrix} \bar{f} x & 0 \\ 0 & 0 \end{pmatrix}, \qquad \psi_\pi(x'\bar{g}) =  \begin{pmatrix} x\bar{f}  & 0 \\ 0 & 0 \end{pmatrix}.$$
Together with \eqref{eq:i}, this shows that 
\begin{equation} \label{eq:iii}
\begin{pmatrix} \Delta & 0 \\ 0 & 0 \end{pmatrix} \subseteq \psi_\pi(\Delta).
\end{equation}
If $x,x' \in \Delta$ are such that \eqref{eq:ii} holds, then
$$\psi_\pi(x'a) = \begin{pmatrix} 0 & 0 \\ x & 0 \end{pmatrix},  \qquad \psi_\pi(ax') = \begin{pmatrix} 0 & x \\ 0 & 0 \end{pmatrix}, \qquad \psi_\pi(ax'a) = \begin{pmatrix} 0 & 0 \\ 0 & x \end{pmatrix}.$$
Together with \eqref{eq:iii}, this completes the proof.\end{proof}

\begin{lemma} \label{Lemma5}  $\mathrm{dim}(B/\Delta^2) \le \big|\cG: [K,K]\big| < \infty$.
\end{lemma}

\begin{proof} Let $\Delta'$ be the ideal in $B$ generated by the set $\{\bar{k}-1 \mid k \in [K,K]\}$. Exactly as in the argument above, showing that the dimension of $B/\Delta$ is at most $|\cG:K|=16$, we see that the dimension of $B/\Delta'$ is at most $|G:[K,K]|$. Now, $K$ is finitely generated, and so is the quotient $K/[K,K]$, which, moreover, is an abelian torsion group. Hence $K/[K,K]$ is finite, so $|\cG:[K,K]| = |\cG:K| |K:[K,K]|$ is finite. For all $k_1,k_2 \in K$,
$$[k_1,k_2] -1 = k_1^{-1}k_2^{-1} \big((k_1-1)(k_2-1)-(k_2-1)(k_1-1)\big) \in \Delta^2,$$
which shows that $\Delta' \subseteq \Delta^2$. This proves the lemma.
\end{proof}

\noindent One more property of $\cG$, that we are going to exploit, is the so-called \emph{contracting property}, already used in \cite{Gr80}. Let $|g|$ denote the length of $g \in \cG$ with respect to the canonical generating set $\{a,b,c,d\}$. With $\psi\colon \cG \to (\cG \times \cG) \rtimes \Z/2\Z$ as defined in \eqref{eq:psi-G}, and $g \in \cG$, we have $\psi(g) = (g_0,g_1) \eta$, where $g_1,g_2 \in \cG$ and $\eta \in \{e, \epsilon\}$. By \cite{Gr80}, see also \cite[Lemma 3.1]{Gr05},
\begin{equation} \label{eq:contracting}
|g_i| \le \frac{|g|+1}{2}, 
\end{equation}
for $i=0,1$. In particular, $|g_i| < |g|$ if $|g| \ge 2$. The set of elements $g \in \cG$ for which $|g| \le 1$ is equal to $N = \{1,a,b,c,d\}$, which is called the \emph{nucleus} of $\cG$. 

We can repeat this process and obtain for each $g \in \cG$ and $v \in \{0,1\}^{n}$ a section $g_v  = g|_v \in \cG$ (defined underneath \eqref{eq:self-similar}), such that $\psi(g_v) = (g_{0v},g_{1v}) \eta_v$, where $\eta_v \in \{e,\epsilon\}$ and $|g_{iv}| \le (|g_v|+1)/2$, for $i=0,1$.  It follows that, for each $g \in \cG$, there exists $n \ge 1$ such that $g|_v \in N$, for all $v$ in $\{0,1\}^{n}$. By the construction of the self-similarity map $\psi_\pi \colon B \to M_2(B)$, this leads to the following:

\begin{lemma} \label{Lemma3} For each $x$ in $B$, there exists $n \ge 1$ such that the $2^n \times 2^n$ matrix $\psi^n_\pi(x) \in M_{2^n}(B)$ has entries in the linear span of the element in the nucleus $\bar{N} = \{1, \bar{a}, \bar{b}, \bar{c}, \bar{d}\}$.
\end{lemma}

\noindent Next we will prove:

\begin{lemma} \label{Lemma4}
Let $J$ be a non-zero ideal in $B$. There is $m \ge 1$ so that $M_{2^m}(\Delta^2) \subseteq \psi_\pi^m(J)$.
\end{lemma}

\begin{proof} Let $x$ be a non-zero element in $J$. Suppose that there exists $m \ge 1$ such that one of the $2^m \times 2^m$ entries, say the $(s,t)$th entry, of $\psi_\pi^m(x)$ is a non-zero scalar $\lambda$. Denote by $e_{ij}^{(m)}$,  $i,j=1, 2, \dots, 2^m$, the standard matrix units of $M_{2^m}(\C)$. Then, upon identifying $M_{2^m}(B)$ with $B \otimes M_{2^m}$, we have
\begin{equation} \label{eq:15}
(p \otimes e_{is}^{(m)})\psi_\pi^{m}(x) (q \otimes e_{tj}^{(m)}) = \lambda pq \otimes e_{ij}^{(m)},
\end{equation}
for all $p,q \in B$ and all $i,j = 1,2, \dots, 2^{m}$. It follows from \eqref{eq:15} and from Lemma~\ref{Lemma2} that $pq \otimes e_{ij}^{(m)}$ belongs to $\psi_\pi^{m}(J)$, for all $p,q \in \Delta$. We conclude that $z \otimes e_{ij}^{(m)}$ belongs to $\psi_\pi^{m}(J)$, for all $z \in \Delta^2$ and all $i,j=1,2, \dots, 2^m$, and hence  that $M_{2^{m}}(\Delta^2) \subseteq \psi_\pi^{m}(J)$.

To complete the proof, we show below that one of the entries of $\psi_\pi^{m}(x)$ is a non-zero scalar,  for some $m \ge 1$. 

Let $n \ge 1$ be as in Lemma~\ref{Lemma3} (associated with our given $x \in B$). 
Write $\psi_\pi^n(x) = (x_{s,t})_{s,t=1}^{2^n}$ with $x_{s,t} \in B$. By the choice of $n$, we deduce that $x_{s,t}$ belongs to the span of $\bar{N} = \{1, \bar{a}, \bar{b}, \bar{c}, \bar{d}\}$,
for all $s,t$. Since $\psi_\pi^n$ is injective, $\psi_\pi^n(x)$ is non-zero, so we can find $s,t$ such that $x_{s,t}$ is non-zero. Write
$$x_{s,t} = \rho \cdot 1 + \xi \bar{a} + \beta \bar{b} + \gamma \bar{c} + \delta \bar{d},$$
for suitable $\rho, \xi, \beta, \gamma, \delta \in \C$. Observe that, by \eqref{eq:8},
\begin{equation} \label{eq:14}
\psi_\pi(x_{s,t}) = \begin{pmatrix} (\beta + \gamma)\bar{a} + \delta+\rho & \xi \\ \xi & \beta \bar{c} + \gamma \bar{d} + \delta \bar{b} + \rho \end{pmatrix}.
\end{equation}
The proof is now divided into three  cases:

1). Assume that $\xi \ne 0$. In this case both off diagonal entries of $\psi_\pi(x_{s,t})$ are non-zero scalars, and since $\psi_\pi(x_{s,t})$ is a sub-matrix of the $2^{n+1} \times 2^{n+1}$ matrix $\psi^{n+1}_\pi(x)$, at least one of the entries of $\psi^{n+1}_\pi(x)$  is a non-zero scalar. 

2). Assume that either $\beta + \gamma \ne 0$, or $\delta+\rho \ne 0$. Use \eqref{eq:8} to compute the $2 \times 2$ matrix
$$\psi_\pi((\beta + \gamma)\bar{a} + \delta+\rho) = \begin{pmatrix} \delta +\rho & \beta + \gamma \\ \beta + \gamma & \delta + \rho \end{pmatrix}.$$
By assumption, one of the scalar entries in this matrix is non-zero. Further, it is a sub-matrix of the $4 \times 4$ matrix $\psi_\pi^2(x_{s,t})$ and hence a sub-matrix of the $2^{n+2} \times 2^{n+2}$ matrix $\psi^{n+2}_\pi(x)$. Thus at least one of the matrix entries of $\psi^{n+2}_\pi(x)$ is a non-zero scalar.

3). Assume that $\xi = \beta + \gamma = \delta +\rho = 0$. Then 
$$\psi_\pi(\beta \bar{c} + \gamma \bar{d} + \delta \bar{b} + \rho) = 
\begin{pmatrix} (\beta+\gamma) \bar{a} + \delta + \rho & 0 \\ 0 & \beta \bar{c} + \gamma \bar{d} + \delta \bar{b} + \rho  \end{pmatrix} =\begin{pmatrix} 0 & 0 \\ 0 & \beta \bar{c} - \beta \bar{d} + \delta \bar{b} - \delta  \end{pmatrix},$$
and
$$\psi_\pi(\beta \bar{c} - \beta \bar{d} + \delta \bar{b} - \delta) = 
\begin{pmatrix} (\beta+\delta) \bar{a} + \beta +\delta  & 0 \\ 0 & \beta \bar{d} -\beta \bar{b} + \delta \bar{c} -\delta  \end{pmatrix}.$$
If $\beta+\delta \ne 0$, then,  as in step 2), $\psi_\pi( (\beta+\delta) \bar{a} + \beta +\delta)$ is a non-zero scalar $2 \times 2$ matrix, which is a sub-matrix of the $16 \times 16$ matrix $\psi_\pi^4(x_{s,t})$, whence at least one of the entries of $\psi_\pi^{n+4}(x)$ is a non-zero scalar. 

If $\beta+\delta=0$, then $\beta \ne 0$ (because $x_{s,t} \ne 0$), so
$\beta \bar{c} - \beta \bar{d} + \delta \bar{b} - \delta = \beta(\bar{c} - \bar{d} - \bar{b} +1)$,
and
$$\psi_\pi(\bar{c} - \bar{d} - \bar{b} +1) = 
\begin{pmatrix} 0 & 0 \\ 0 & \bar{d}-\bar{b}-\bar{c}+1 \end{pmatrix}, \quad 
\psi_\pi(\bar{d}-\bar{b}-\bar{c}+1) = 
\begin{pmatrix} 2-2\bar{a} & 0 \\ 0 & \bar{b}-\bar{c}-\bar{d}+1\end{pmatrix}.
$$
Arguing as in step 2), we see that $\psi_\pi(2-2\bar{a})$ is a non-zero scalar $2 \times 2$ matrix, which is a sub-matrix of the $32\times 32$ matrix $\psi_\pi^5(x_{s,t})$, so  at least one of the entries of $\psi_\pi^{n+5}(x)$ is a non-zero scalar. 
\end{proof}

\noindent We are now ready to complete the proof of Theorem~\ref{thm:B}.

\medskip
\noindent \emph{Proof of Theorem~\ref{thm:B} {\rm{(iv)}}}: Let $J$ be a non-zero ideal in $B$. Use Lemma~\ref{Lemma4} to find $n \ge 1$ such that $M_{2^n}(\Delta^2) \subseteq \psi_\pi^n(J)$. Since $\psi_\pi^n$ is injective, it follows that 
$$\dim(B/J) = \dim(\psi_\pi^n(B)/\psi_\pi^n(J)) \le \dim(M_{2^n}(B/\Delta^2)) = 2^{2n} \dim(B/\Delta^2) < \infty,$$
by Lemma~\ref{Lemma5}. This completes the proof.

\vspace{.3cm} \noindent We end our paper by showing that if $G$ is a residually finite group for which $\C[G]$ is $^*$-just-infinite, then $G$ is hereditarily just-infinite (see also the discussion at the end of Section~\ref{sec:ji-groups}). Indeed, if $\C[G]$ is $^*$-just-infinite, then $G$ is just-infinite, by Corollary~\ref{cor:3conditions}. By the trichotomy for just-infinite groups, \cite[Section 6]{Gr00}, $G$ must be either a branch group or hereditarily just-infinite, and the theorem below rules out the former possibility.

We remind the reader about some facts concerning branch groups (see also \cite{Gr00}). Consider a spherically homogeneous rooted tree $T = T_{\bar{m}}$, where $\bar{m} = \{m_n\}_{n=0}^\infty$, is the branching index of the tree (each $m_n \ge 2$ is an integer). For each vertex $v$  in the $k$th level of the tree $T$, let $T_v$ be the sub-tree of $T$ consisting of all vertices ``below'' $v$, so that $T_v$ is a rooted tree with root $v$ and branching index $\{m'_n\}_{n=0}^\infty$, where $m'_n = m_{n+k}$. 

Suppose that $G$ is a group that acts on such a spherically homogeneous rooted tree $T$. Then $G$ fixes the root of the tree and hence leaves each level of the tree invariant. The \emph{rigid stabilizer} of a vertex $v \in T$, denoted by $\mathrm{rist}_G(v)$,  is the subgroup of $G$ consisting of all $g \in G$ which act trivially outside $T_v$ (and fix $v$). The rigid stabilizer, $\mathrm{rist}_G(n)$, at level $n \in \N$ is the subgroup of $G$ generated by the rigid stabilizers $\mathrm{rist}_G(v)$ of all vertices $v$ at level $n$. It is easy to see that $\mathrm{rist}_G(n)$ is, in fact, the direct product of the groups $\mathrm{rist}_G(v)$, where $v$ is a vertex at level $n$.

A group $G$ is said to be a \emph{branch group} if it admits a faithful action on such a spherically homogeneous rooted tree $T = T_{\bar{m}}$, such that the index $|G: \mathrm{rist}_G(n)|$ is finite, for all  $n \in \N$, and such that $T$ acts transitively on each level of the tree.

\begin{theorem} \label{thm:no-branch} If $G$ is a branch group, then $\C[G]$ is not $^*$-just-infinite, whence $C^*(G)$ is not just-infinite.
\end{theorem}

\begin{proof} Fix an action of $G$ on a spherically homogeneous rooted tree $T = T_{\bar{m}}$ satisfying the above mentioned conditions. Let $\pi$ be the Koopman representation of $G$ into the unitary group of the Hilbert space $H = L^2(\partial T, \mu)$, where $\mu = \bigtimes_{n=0}^\infty \mu_n$, and $\mu_n$ is the uniform probability measure on the set $\{1,2,\dots, m_n\}$. Denote also by $\pi$ the associated $^*$-representation $\C[G] \to B(H)$. 

We show that $\pi \colon \C[G] \to B(H)$ is not injective, and that $\pi(\C[G])$ is infinite dimensional. This will imply that $\C[G]$ is not $^*$-just-infinite, and hence (by Corollary~\ref{cor:3conditions}) that $C^*(G)$ is not just-infinite. Since $G$ acts level transitively on $T$, we conclude that $G$ is infinite and that $\pi(\C[G])$ is infinite dimensional.

Let $m= m(0)$ and let $v_1,v_2, \dots, v_m$ be the vertices at the first level of the tree $T$ (below the root of the tree). 
The condition that  $|G: \mathrm{rist}_G(1)|$ is finite  implies that $\mathrm{rist}_G(1)$, which is isomorphic to $\bigtimes_{j=1}^m \mathrm{rist}_G(v_j)$, is infinite. Moreover, by level transitivity of the action of $G$ on $T$, the rigid stabilizers  $\mathrm{rist}_G(v_j)$ are pairwise conjugate, so they are, in particular, non-trivial. We can therefore choose $g_j \in \mathrm{rist}_G(v_j)$, for $j=1,2$, such that $g_j \ne 1$. Observe that $(1-g_1)(1-g_2) = 1 - g_1-g_2 +g_1g_2$ is non-zero in $\C[G]$, because $g_1 \ne 1$ and $g_2 \ne 1$. 

For $i=1,2, \dots, m$, let $X_i$ be the subset of $\partial T$ consisting of words  that start with $v_i$, i.e., $X_i =\partial T_{v_i}$, so that $\partial T$ is the disjoint union of the sets $X_1, X_2, \dots X_m$. Set $H_i = L^2(X_i,\mu)$. Then  $H = \bigoplus_{i=1}^m H_i$. Let $P_i$ be the projection from $H$ onto $H_i$. Since $g_j$ acts trivially on the sub-trees $T_{v_i}$, for $i \ne j$, we conclude that $P_i$ commutes with $\pi(g_j)$ for $i=1,2, \dots, m$ and $j=1,2$, 
and $P_i\pi(g_j) = P_i$, when $j \ne i$. Hence $\pi(1-g_j)P_i = 0$, for $i \ne j$. It follows that
$$\pi((1-g_1)(1-g_2)) = \pi((1-g_1)(1-g_2))\sum_{i=1}^m P_i =  \pi((1-g_1)(1-g_2))P_2 = \pi(1-g_1)P_2\pi(1-g_2) = 0,$$
so $\pi \colon \C[G] \to B(H)$ is not injective, as wanted.
\end{proof}

\noindent The theorem above (and its proof) contains item (iii) of Theorem~\ref{thm:B}, since $\cG$ is a branch group. As in the conclusion of Theorem~\ref{thm:B}, it can happen, at least for some just-infinite branch groups $G$ (for example, when $G = \cG$), that $\pi(\C[G])$ is just-infinite. It may also happen, for some just-infinite branch groups $G$, that $C^*_\pi(G)$ is a RFD just-infinite \Cs, where $\pi$ as above is the Koopman representation of $G$ arising from its action on a tree. We conjecture that $C^*_\pi(\cG)$ is a RFD just-infinite \Cs.

{\small
{
\bibliographystyle{amsplain}
\providecommand{\bysame}{\leavevmode\hbox to3em{\hrulefill}\thinspace}
\providecommand{\MR}{\relax\ifhmode\unskip\space\fi MR }
\providecommand{\MRhref}[2]{%
  \href{http://www.ams.org/mathscinet-getitem?mr=#1}{#2}
}
\providecommand{\href}[2]{#2}

\vspace{1cm}

\noindent Rostislav Grigorchuk\\
Mathematics Department\\
Texas A\& M University\\ 
College Station, TX 77843-3368\\
USA\\
grigorch@math.tamu.edu.
\\

\noindent Magdalena Musat\\
Department of Mathematical Sciences\\
University of Copenhagen\\ 
Universitetsparken 5, DK-2100, Copenhagen \O\\
Denmark \\
musat@math.ku.dk\\

\noindent Mikael R\o rdam \\
Department of Mathematical Sciences\\
University of Copenhagen\\ 
Universitetsparken 5, DK-2100, Copenhagen \O\\
Denmark \\
rordam@math.ku.dk\\

\end{document}